\documentclass[12pt]{amsart}
\usepackage{amscd,amssymb,amsthm,amsmath,amssymb,textcomp,multirow,rotating,longtable,mathrsfs,fancyhdr,textgreek,mathtools,wasysym,pdflscape,imakeidx,makecell}
\usepackage[matrix,arrow,curve]{xy}
\usepackage{tikz}
\usepackage{multicol}

\newcommand{\DR}{\mathbb{R}} 
\newcommand{\DC}{\mathbb{C}}
\newcommand{\DZ}{\mathbb{Z}}
\newcommand{\DQ}{\mathbb{Q}}
\newcommand{\DP}{\mathbb{P}}

\newcommand{\MO}{\mathcal{O}}

\newcommand{\MM}{\mathcal{M}}

\newcommand{\Aut}{\mathrm{Aut}}
\newcommand{\pr}{\mathrm{pr}}

\newcommand{\PGL}{\mathrm{PGL}}

\makeatletter
\newcommand*{\da@rightarrow}{\mathchar"0\hexnumber@\symAMSa 4B }
\newcommand*{\da@leftarrow}{\mathchar"0\hexnumber@\symAMSa 4C }
\newcommand*{\xdashrightarrow}[2][]{%
  \mathrel{%
    \mathpalette{\da@xarrow{#1}{#2}{}\da@rightarrow{\,}{}}{}%
  }%
}
\newcommand{\xdashleftarrow}[2][]{%
  \mathrel{%
    \mathpalette{\da@xarrow{#1}{#2}\da@leftarrow{}{}{\,}}{}%
  }%
}
\newcommand*{\da@xarrow}[7]{%
  \sbox0{$\ifx#7\scriptstyle\scriptscriptstyle\else\scriptstyle\fi#5#1#6\m@th$}%
  \sbox2{$\ifx#7\scriptstyle\scriptscriptstyle\else\scriptstyle\fi#5#2#6\m@th$}%
  \sbox4{$#7\dabar@\m@th$}%
  \dimen@=\wd0 %
  \ifdim\wd2 >\dimen@
    \dimen@=\wd2 %
  \fi
  \count@=2 %
  \def\da@bars{\dabar@\dabar@}%
  \@whiledim\count@\wd4<\dimen@\do{%
    \advance\count@\@ne
    \expandafter\def\expandafter\da@bars\expandafter{%
      \da@bars
      \dabar@ 
    }%
  }%
  \mathrel{#3}%
  \mathrel{%
    \mathop{\da@bars}\limits
    \ifx\\#1\\%
    \else
      _{\copy0}%
    \fi
    \ifx\\#2\\%
    \else
      ^{\copy2}%
    \fi
  }%
  \mathrel{#4}%
}
\makeatother

\newcommand*\circled[1]{\tikz[baseline=(char.base)]{
            \node[shape=circle,draw,inner sep=2pt] (char) {#1};}}

\newtheorem{theorem}{Theorem}[section]

\newtheorem{lemma}[theorem]{Lemma}
\newtheorem{corollary}[theorem]{Corollary}
\newtheorem*{corollary*}{Corollary}
\newtheorem*{maincorollary*}{Main Corollary}

\newtheorem*{conjecture*}{Conjecture}

\newtheorem*{problem*}{Problem}
\newtheorem*{calabiproblem*}{Calabi Problem}
\newtheorem*{Main Theorem*}{Main Theorem}

\newtheorem*{theorem*}{Theorem}
\newtheorem*{maintheorem*}{Main Theorem}

\theoremstyle{definition}

\newtheorem*{example*}{Example}
\newtheorem{definition}[theorem]{Definition}

\theoremstyle{remark}
\newtheorem{remark}[theorem]{Remark}
\newtheorem*{remark*}{Remark}

\makeatletter\@addtoreset{equation}{subsection} \makeatother

\makeindex

\title{On $K$-stability of  $\DP^3$ blown up along the disjoint union of  a twisted cubic curve and a line}

    \author{Elena Denisova}

\begin{document}
\maketitle

\begin{abstract}
    We find all $K$-polystable smooth Fano threefolds that can be obtained as blowup of $\DP^3$ along the disjoint union of a twisted cubic curve and a line.
\end{abstract}

\tableofcontents

\section{Introduction}
\noindent  Smooth Fano threefolds defined over the field $\DC$ have been classified in \cite{Is77,Is78,MoMu81,MoMu03} into $105$ families. The detailed description of these families can be found in \cite{Fano21}.  In \cite{Fano21} the following problem was  posed: \begin{calabiproblem*}
Find all $K$-polystable smooth Fano threefolds in each family.
\end{calabiproblem*}
 \noindent  This problem was solved for $72$ families. A great contribution to solving this problem was made by the authors of \cite{Fano21}. After their work only $34$ families were left. The same year I. Cheltsov and J. Park obtained the result for one more family \cite{FanoR2D30}.
 \noindent 
Suppose $X$ is a~general member of the~family \textnumero $\mathscr{N}$, then \cite[Main~Theorem]{Fano21} tells us that
$$
X\ \text{is } K\text{-polystable} \iff\mathscr{N}\not\in\left\{\aligned
&\ 2.23, 2.26, 2.28, 2.30, 2.31, 2.33, 2.35, 2.36, \\
& \ 3.14, 3.16, 3.18,3.21, 3.22, 3.23, \\
&\ 3.24, 3.26, 3.28, 3.29,  3.30, 3.31, \\
&\    4.5, 4.8, 4.9, 4.10, 4.11, 4.12,\\
&\  5.2
\endaligned
\right\}.
$$
 Suppose $X$ is a member of Family 3.12. Then we describe $X$  as the blowup $\pi:X\to \DP^3$ of $\DP^3$  at a twisted cubic $C$ and line  $L$ that is disjoint from $C$ (see Section \ref{geometry} for explicit description of all members of this family).  These threefolds form a one-dimensional family.  Groups of automorphisms of such threefolds are finite except for one threefold which has automorphism group $\DC^*\rtimes \DZ/2\DZ$. It was shown in \cite[\S 5.18]{Fano21} that this threefold is $K$-polystable. Moreover, it was used to show that the general member of this family is $K$-polystable. Furthermore, in \cite[\S 7.7]{Fano21} it was shown that there exists a non $K$-polystable member in this family and it was conjectured that all other smooth Fano threefolds in Family 3.12 are $K$-polystable. The goal of this work is to prove this conjecture
 and complete the~description of all $K$-polystable smooth Fano threefolds of Picard rank $3$ and degree $28$ started in \cite{Fano21}.
\begin{Main Theorem*}
All the smooth threefolds except one in  Family  3.12 are $K$-polystable.
\end{Main Theorem*}
\noindent  Hence, all smooth Fano threefolds in Family 2.12 except one described in  \cite[\S 7.7]{Fano21} admit a Kähler-Einstein metric.
\subsection{ Plan of the paper} 
In Section 2 we state the results which will use to prove Main Theorem. In Section 3 we will discuss the equivariant  geometry of $\DP^3$ which will help us to understand the  equivariant  geometry of $X$ in Family 3.12. We will focus our attention on the members in Family 3.12 for which the $K$-polystability has not been proved yet. In this section we show that $\Aut(X)\cong \Aut(\DP^3,L+C)$ contains a subgroup $G\cong \DZ/2\DZ\times \DZ/2\DZ$. We will show that there are no $G$-fixed points on $\DP^3$, describe $G$-invariant quadrics containing $C$ on $\DP^3$ and $G$-invariant lines on $\DP^3$. At the end of this section we give description of the Mori cone and the cone of effective divisors on $X$. Finally, in Section 4 we prove our Main Theorem.
\subsection{ Plan of the proof}
If $X$ is not $K$-polystable then it follows from \cite[Corollary~4.14]{Zhuang} that  there exists a~$G$-invariant prime divisor $F$ over $X$
such that $\beta(F)\leqslant 0$ where $\beta(F)$ was defined in \cite{Fujita2019}, see also \cite[Definition 1.2.1]{Fano21} and Section 2.
Let $Z$ be the center of $F$ on $X$. Then $Z$ is not a~point since $X$ has no $G$-fixed points, and $Z$ is not a~surface by \cite[Theorem~10.1]{Fujita2016},
so that $Z$ is a~$G$-invariant irreducible curve. Then we derive a contradiction as  follows:
\begin{enumerate}
    \item[$\circled{1}$] We use Abban-Zhuang theory (see \cite{AbbZh2020}) and its corollary \cite[Collolary 1.7.26]{Fano21} to exclude the case when $\pi(Z)$ is a line such that $\pi(Z)\ne L$ and $\pi(Z)\cap C=\emptyset$. This is done in Lemma \ref{lemma:3-12-Z-lines}.
    \item[$\circled{2}$] We use Abban-Zhuang theory (\cite{AbbZh2020}) and its corollary \cite[Collolary 1.7.26]{Fano21} to exclude the case when $\pi(Z)\subset L$.
    This is done in Lemma \ref{lemma:3-12-Z-in-E}.
    \item[$\circled{3}$] We use Abban-Zhuang theory(see \cite{AbbZh2020}) and its corollary \cite[Collolary 1.7.26]{Fano21} to show that $\pi(Z)$ is not contained in a $G$-invariant quadric passing through $C$. This is done in Lemma \ref{lemma:3-12-Z-in-Q}.
    \item[$\circled{4}$] Using $\circled{1}$, $\circled{2}$, $\circled{3}$ we deduce that $\pi(Z)$ is not a line.
    \item[$\circled{5}$] It follows from $\beta(F)\le 0$ that $Z$ is contained in $\mathrm{Nklt}(X,\lambda D)$ for some $G$-invariant effective $\DQ$-divisor $D\sim_{\DQ} -K_X$ and $\lambda\in \DQ$ such that $\lambda<\frac{3}{4}$. Moreover it follows from $\circled{2}$, $\circled{3}$ and description of the cone of effective divisors on $X$ that $Z$ is not contained in the surface in $\mathrm{Nklt}(X,\lambda D)$. See Corollary \ref{pi(Z)isnotsurf}. 
    \item[$\circled{6}$] Using $\circled{5}$ we derive that $\pi(Z)\not\subset C$. 
    This is done in Corollary \ref{corollary:3-12-EC}.
    \item[$\circled{7}$] Finally, we use $\circled{6}$ to show that $\pi(Z)$ is a line in $\DP^3$, which contradicts $\circled{4}$.
\end{enumerate}
\section{Preliminary results}
\noindent 
Let $X$ be a~Fano variety with Kawamata log terminal singularities,
let $G$ be a~reductive subgroup in $\mathrm{Aut}(X)$,
let $f\colon\widetilde{X}\to X$ be a~$G$-equivariant birational morphism,
let $F$ be a~$G$-invariant prime divisor in $\widetilde{X}$, and let $n=\mathrm{dim}(X)$.
\begin{definition}
\label{definition:divisor-over-X}\index{divisor over $X$}\index{dreamy divisor}
We say that $F$ is a~$G$-invariant prime divisor \emph{over} the~Fano variety $X$.
If~$F$~is~\mbox{$f$-exceptional}, we say that $F$ is an~exceptional $G$-invariant prime divisor \emph{over}~$X$.
We will denotes the~subvariety $f(F)$ by $C_X(F)$.
\end{definition}
\noindent Let
\index{$S_X(F)$}
$$
S_X(F)=\frac{1}{(-K_X)^n}\int_{0}^{\tau}\mathrm{vol}(f^*(-K_X)-xF)dx,
$$
where $\tau=\tau(F)$ is the~pseudo-effective threshold of $F$ with respect to $-K_X$, i.e. we have \index{pseudo-effective threshold}
\begin{equation}
\tau(F)=\mathrm{sup}\Big\{ u \in \mathbb{Q}_{>0}\ \big\vert \ f^*(-K_X)-uF\ \text{is big}\Big\}.
 \label{eq:1}
\end{equation}
Let $\beta(F)=A_X(F)-S_X(F)$,
where $A_X(F)$ is the~log discrepancy of the~divisor $F$.\index{$\beta$-invariant}
\begin{theorem}[{\cite[Corollary~4.14]{Zhuang}}]
\label{theorem:G-K-polystability}
Suppose that $\beta(F)>0$ for every $G$-invariant  prime divisor $F$ over $X$.
Then $X$ is K-polystable.
\end{theorem}

\begin{theorem}[{\cite[Theorem~10.1]{Fujita2016}}]
\label{theorem:divisorially-stable}
Let $X$ be any~smooth Fano threefold that is not contained in the~following $41$ deformation families:
\begin{center}
\textnumero 1.17, \textnumero 2.23, \textnumero 2.26, \textnumero 2.28, \textnumero 2.30, \textnumero 2.31, \textnumero 2.33, \textnumero 2.34, \textnumero 2.35, \textnumero 2.36,\\
\textnumero 3.9, \textnumero 3.14, \textnumero 3.16, \textnumero 3.18, \textnumero 3.19, \textnumero 3.21, \textnumero 3.22, \textnumero 3.23, \textnumero 3.24, \textnumero 3.25, \\
\textnumero 3.26, \textnumero 3.28, \textnumero 3.29, \textnumero 3.30, \textnumero 3.31, \textnumero 4.2, \textnumero 4.4, \textnumero 4.5, \textnumero 4.7, \textnumero 4.8, \textnumero 4.9,\\
\textnumero 4.10, \textnumero 4.11, \textnumero 4.12, \textnumero 5.2, \textnumero 5.3, \textnumero 6.1, \textnumero 7.1, \textnumero 8.1, \textnumero 9.1, \textnumero 10.1.
\end{center}
Then $S_X(Y)<1$ for every irreducible surface $Y\subset X$, i.e. $X$ is divisorially stable.
\end{theorem}
\begin{theorem}[{\cite[Collolary 1.7.26]{Fano21}}]
\label{corollary:Kento-formula-Fano-threefold-surface-curve}
Let $X$ be a~smooth Fano threefold, let $Y$ be an irreducible normal surface in the~threefold $X$,
let $Z$ be an irreducible curve in $Y$, and let $F$~be~a~prime divisor over the~threefold $X$ such that $C_X(F)=Z$. Then
\begin{equation}
\label{equation:Kento-formula-Fano-threefold-surface-curve}
\frac{A_X(F)}{S_X(F)}\geqslant\min\Bigg\{\frac{1}{S_X(Y)},\frac{1}{S\big(W^Y_{\bullet,\bullet};Z\big)}\Bigg\}
\end{equation}
and
$$\hspace*{-0.3cm}
S\big(W^Y_{\bullet,\bullet};Z\big)=\frac{3}{(-K_X)^3}\int_0^\tau\big(P(u)^{2}\cdot Y\big)\cdot\mathrm{ord}_{Z}\Big(N(u)\big\vert_{Y}\Big)du+\frac{3}{(-K_X)^3}\int_0^\tau\int_0^\infty \mathrm{vol}\big(P(u)\big\vert_{Y}-vZ\big)dvdu,
$$
where $P(u)$ is the~positive part of the~Zariski decomposition of the~divisor $-K_X-uY$, and $N(u)$ is its negative part.
\end{theorem}
\begin{lemma}[{\cite[Lemma 1.4.4]{Fano21}}]\label{lemma144}
Let $X$ be a~Fano variety with at most Kawamata log terminal singularities of dimension $n\geqslant 2$, $Z$ be a~proper irreducible subvariety in $X$. Let $f\colon\widetilde{X}\to X$ be an~arbitrary $G$-equivariant birational morphism,
let $F$ be a~$G$-invariant prime divisor in $\widetilde{X}$ such that $Z\subseteq f(F)$,
and let $\tau(F)$ satisfy \ref{eq:1}. Suppose in addition that $X$ is smooth and $\mathrm{dim}(Z)\geqslant 1$.
Then
$$
\frac{A_X(F)}{S_X(F)}>\frac{n+1}{n}\alpha_{G,Z}(X),
$$
where $$
\alpha_{G,Z}(X)=\mathrm{sup}\left\{\lambda\in\mathbb{Q}\ \left|\ %
\aligned
&\text{the~pair}\ \left(X, \lambda D\right)\ \text{is log canonical at general point of $Z$ for any}\\
&\text{effective $G$-invariant $\mathbb{Q}$-divisor}\ D\ \text{on}\ X\ \text{such that}\ D\sim_{\mathbb{Q}} -K_X\\
\endaligned\right.\right\}.%
$$
\end{lemma}
\begin{lemma}[{\cite[Corollary A.13]{Fano21}}]\label{CorollaryA13}
Suppose $X=\mathbb{P}^3$ and $B_X\sim_{\mathbb{Q}} -\lambda K_X$
for some rational number $\lambda<\frac{3}{4}$.
Let $Z$ be the~union of one-dimensional components of  $\mathrm{Nklt}(X,B_X)$.
Then $\mathcal{O}_{\mathbb{P}^3}(1)\cdot Z\leqslant 1$.
\end{lemma}
\begin{lemma}[{\cite[Corollary A.15]{Fano21}}]\label{CorollaryA15}
Suppose that $X$ is a~smooth Fano threefold,  $-K_{X}$ is nef and big, $B_X\sim_{\mathbb{Q}} -\lambda K_X$ for some rational number~\mbox{$\lambda<1$},
and there exists a~surjective morphism with connected fibers \mbox{$\phi\colon X\to \mathbb{P}^1$}.
Set $H=\phi^*(\mathcal{O}_{\mathbb{P}^1}(1))$. Let $Z$ be the~union of one-dimensional components of $\mathrm{Nklt}(X,\lambda B_X)$.
Then $H\cdot Z\leqslant 1$.
\end{lemma}
\section{Geometry of Fano Threefolds in Family  \textnumero 3.12}
\label{geometry}
\subsection{Basic properties  of Fano Threefolds in  Family  \textnumero 3.12} \label{6}
\noindent Let $C$ be the smooth twisted cubic curve in $\DP^3$ that is the image of the map $\DP^1\hookrightarrow \DP^3$ given by
$$[x:y]\to[x^3:x^2y:xy^2:y^3]$$
let $L$ be a line in $\DP^3$ that is disjoint from $C$, and let $\pi:X\to \DP^3$ be the blow up of $\DP^3$ along $C$ and $L$. Then   $X$ is {\it a Fano threefold in family 3.12} and all threefolds in this family can be obtained this way. Note that there exists the following commutative diagram:
$$
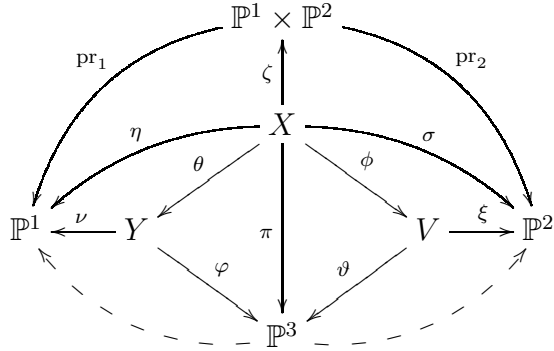
\xymatrix{
&&\DP^1\times \DP^2\ar@/^2pc/[ddrr]^{\pr_2}\ar@/^-2pc/[ddll]_{\pr_1}&&\\
&&X \ar@/^1pc/[drr]^{\sigma}\ar@/^-1pc/[dll]_{\eta}\ar[dl]_{\theta}\ar[dr]^{\phi}\ar[dd]_{\pi}\ar[u]^{\zeta}&&\\
\DP^1&Y\ar[dr]^{\varphi}\ar[l]_{\nu}&&V\ar[r]^{\xi}\ar[dl]_{\vartheta}&\DP^2\\
&&\DP^3\ar@{-->}@/^1.4pc/[ull]\ar@{-->}@/^-1.4pc/[urr]&&
}$$
Where:
\begin{itemize}
    \item $\varphi$ is the blowup of a line $L$,
    \item $\vartheta$ is the blowup of a curve $C$,
    \item $\phi$ is the blowup of a curve $\vartheta^* L$,
    \item $\theta$ is the blowup of a curve $\varphi^* C$,
    \item the left dashed arrow is the linear projection from the line $L$,
    \item the right dashed arrow is given by the linear system of quadrics that contain $C$,
    \item $\xi$ is a $\DP^1$-bundle,
    \item $\nu$ is a $\DP^2$-bundle,
    \item $\sigma$ is a non-standard conic bundle,
    \item $\eta$ is a fibration into the del Pezzo surfaces of degree $6$,
    \item $\zeta$ is the contraction of the proper transforms of the quartic surface in $\DP^3$ that is spanned by the secants of the curve $C$ that intersect $L$,
    \item  $\pr_1$ and $\pr_2$ are projections to the first and the second factors, respectively.
\end{itemize}
Let $H$ be a plane in $\DP^3$, $E_L$ be the exceptional surface of $\pi$ that is mapped to $L$, $E_C$ be the exceptional surface of $\pi$ that is mapped to $C$, $R$ be $\zeta$-exceptional surface. Then $$R\sim_{\DQ} \pi^*(4H)-2E_C-E_L,$$
 and $$-K_X\sim_{\DQ}\pi^*(4H)-E_C-E_L .$$
 
\subsection{Construction of $R$}\label{1}
Consider the commutative diagram:
$$\xymatrix{ X\ar[dr]^{\phi}\ar[dd]_{\pi} \ar@/^2pc/[rrdd]^{\sigma}
\\&V\ar[dl]_{\vartheta}\ar[dr]^{\xi}\\
\DP^3\ar@{-->}[rr]&&\DP^2}$$
Where $\xi$ is a $\DP^1$-bundle given by the linear system $|\vartheta^*(2H)-E_C|$, $\vartheta$ is the blowup of $C$ and the dashed arrow is given by the linear system of quadrics containing $C$, $\phi$ is the blowup of $\vartheta^*L$. Denote $\tilde{L}=\vartheta^*L$. What is the image of $\tilde{L}$ in $\DP^2$? We have that 
$$\tilde{L}\cdot(\vartheta^*(2H)-E_C)=2$$
which means that $\xi(\tilde{L})$ is a conic. The preimage of this conic are all the secants of $C$ which intersect $\tilde{L}$.  Therefore, $\pi(R)$ is spanned by secants of $C$ that intersect $L$. Note that the class of the preimage is $$\xi^*(\MO_{\DP^2}(2))=2(\vartheta^*(2H)-E_C)= \vartheta^*(4H)-2E_C.$$
Note that moreover that $\xi(\tilde{L})$ is a smooth conic and $\xi$ is a $\DP^1$-bundle thus the preimage of $\xi(\tilde{L})$ is a smooth surface so it is smooth along $\tilde{L}$ thus the class of $R$ in $\DP^3$ is given by
$$R\sim_{\DQ} \pi^*(4H)-2E_C-E_L.$$
\subsection{$\DZ/2\DZ\times \DZ/2\DZ$-action on $\DP^3$ and Fano threefolds in Family 3.12}
\label{action}
Note that $\Aut(X)\cong \Aut(X,C+L)$. On the other hand, we have 
$$\Aut(\DP^3,C)=\PGL_2(\DC),$$
where $\Aut(\DP^3,C)$ is the group of automorphisms of $\DP^3$ which fix $C$ as a set.  
\subsubsection{Types of threefolds in Family 3.12} \label{types}
We look at the projection from the line $L$ which is disjoint from $C$ to $\DP^1$:
$$\phi_L:\DP^3\xdashrightarrow{} \DP^1,$$
which gives a $3$-cover of $\DP^1$:
$$ \phi_L\vert_C:C\xrightarrow{3:1} \DP^1.$$
Then by Riemann-Hurwitz we have that the degree of the ramification divisor is $4$. The multiplicity in each ramification point is either $2$ or $3$. So we have $3$ options for ramification points: 
\begin{itemize}
    \item there are two ramification points both of multiplicity $3$,\\
    \item there is one ramification point of multiplicity $3$ and two ramification points of multiplicity $2$,\\
    \item there are four ramification points of multiplicity $2$.
\end{itemize}
We see that there are at least two ramification points on $C$. By acting on $C$ by the $\PGL(2,\DC)$ we can make these points to be $p_1=[1:0]$, $p_2=[0:1]$ on $C$. Now we look at the line  $L$. It is the intersection of $2$ planes which are tangent to $C$ at points $p_1$ and $p_2$ (note that these planes are different since the plane intersects the cubic $C$ in three points, so the same plane cannot be tangent to $C$ at two points) so it is given by the equations:
$$L:\begin{cases}x_0=r_1x_1,\\
x_3=r_2x_2.
\end{cases}$$
Now we have $3$ cases:
\\$\circled{1}$ $r_1=r_2=0$ so $L$ is given by the equations:
$$L:\begin{cases}x_0=0,\\
x_3=0.
\end{cases}$$
 Here we have two ramification points of multiplicity $3$. This case was described in \cite{Fano21}. The corresponding threefold $X$ is $K$-polystable in this case.
\\$\circled{2}$ $r_1=0$, $r_2\ne 0$ (which is symmetric to the case $r_1\ne 0$, $r_2= 0$) so $L$ is given by the equations:
$$L:\begin{cases}x_0=0,\\
x_3=r_2x_2.
\end{cases}$$
Using the action of $\DC^*$ by the matrix which fixes $C$:
$$\begin{pmatrix}
1 &0&0&0\\
0 &r_2&0 &0\\
0& 0& r_2^2 &0\\
0& 0&0 & r_2^3
\end{pmatrix}$$
We can assume that $L$ is given by
$$L:\begin{cases}x_0=0,\\
x_3=x_2.
\end{cases}$$
Here we have one ramification point of multiplicity $3$ and two ramification points of multiplicity $2$. This case was described in \cite{Fano21} where it was proved in  that $X$ is not $K$-polystable. 
\\$\circled{3}$ $r_1\ne 0$, $r_2\ne 0$ so $L$ is given by the equations
$$L:\begin{cases}x_0=r_1x_1,\\
x_3=r_2x_2.
\end{cases}$$
Using the action of $\DC^*$ by the matrix which fixes $C$:
$$\begin{pmatrix}
1 &0&0&0\\
0 &\lambda &0 &0\\
0& 0& \lambda^2 &0\\
0& 0&0 & \lambda^3
\end{pmatrix},$$
 where $\lambda$ satisfies $\lambda^2=\frac{r_2}{r_1}$. We can assume that $L$ is given by
$$L:\begin{cases}x_0=rx_1,\\
x_3=rx_2.
\end{cases}$$
Note that:
\begin{itemize}
    \item $r\ne 0$ since otherwise we are in case $\circled{1}$,
    \item $r\ne \pm 1$ since otherwise $L$ intersects $C$ which is prohibited,
    \item $r\ne \pm 3$ since otherwise there exists a plane containing $L$ which is tangent to $C$ with multiplicity $3$ (it is a plane given by $x_3+3x_3+3x_1+x_0=0$ in case $r=-3$ and a plane $-x_3+3x_3-3x_1+x_0=0$ in case $r=3$) so this case is projectively  isomorphic to the case $\circled{2}$.
\end{itemize}
Now the involution on $\DP^3$ given by $[x_0:x_1:x_2:x_3]\to [x_3:x_2:x_1:x_0]$ fixes $C$ an $L$. We can do it for any pair of four ramification points on $C\cong \DP^1$. This gives the action of $\DZ/2\DZ\times\DZ/2\DZ$. More precisely this group is generated by the involutions viewed on $\DP^1$:
$$\begin{pmatrix}
0 & 1\\
1 & 0
\end{pmatrix}\text{ }\text{ and }\text{ }
\begin{pmatrix}
1 & -\frac{r(r^2 - 5 + (r^4 - 10r^2 + 9)^{1/2})}{2(r^2 - 3 +(r^4 - 10r^2 + 9)^{1/2})}\\
\frac{r^2+3+(r^4-10 r^2+9)^{1/2}}{4 r} & -1
\end{pmatrix}
$$
The action on $\DP^3$ is given by the map induced by $[x:y]\to[x^3:x^2y:xy^2:y^2]$.
\subsubsection{$\DZ/2\DZ\times \DZ/2\DZ$-fixed points on $X$}
From now on we assume until the end of this section that we are in case $\circled{3}$ of the previous part and $G=\DZ/2\DZ\times \DZ/2\DZ$. In particular, $\Aut(X)$ is finite.
\begin{lemma}\label{fixedpoints}
There are no $G$-invariant planes in $\DP^3$
\end{lemma}
\begin{proof}
Note that $G\hookrightarrow \Aut(C)$ since $C$ is a spatial curve.  If there exists a  $G$-invariant plane $\Pi$ consider the intersection of $\Pi$ with $C$. There are three points in $\Pi\cap C$ counted with multiplicities. Thus, since the order of $G$ is $4$ then there is a $G$-fixed point on $C\cong \DP^1$, which is a contradiction.  
\end{proof}
\begin{corollary}
There are no $G$-fixed points in $\DP^3$.
\end{corollary}
\begin{corollary}
The threefold $X$ does not contain $G$-invariant points.
\end{corollary}
 \subsubsection{$\DZ/2\DZ\times \DZ/2\DZ$-invariant Quadrics Containing $C$} 
\label{2}
Let $\mathcal{M}$ be the linear system of quadrics in $\DP^3$ that contain the curve $C$.
\begin{lemma}\label{lemmasystemofquadrics}
The linear system $\MM$ is $3$-dimensional, it contains exactly $3$ $G$-invariant surfaces, and these surfaces are smooth.
\end{lemma}
\begin{proof}
Note that all this statement does not depend on the equivariant choice of coordinates. We know that groups isomorphic to $\DZ/2\DZ\times \DZ/2\DZ$ are conjugate in $\PGL_2(\DC)$ so we can choose coordinates such that   the generators of our group will look like:
$$\tau_1:[x:y]\to[y:x],$$
$$\tau_2:[x:y]\to[x:-y].$$
Which gives us the action on $\DP^3$ by:
$$\tau_1:[x_0:x_1:x_2:x_3]\to [x_3:x_2:x_1:x_0],$$
$$\tau_2:[x_0:x_1:x_2:x_3]\to [x_0:-x_1:x_2:-x_3].$$
The linear system  $\MM$ is clearly 3-dimensional. We can provide the equations for $3$ $G$-invariant quadrics containing $C$:
$$Q_1:x_0x_3=x_1x_2,\text{ }\text{ }\text{ }Q_2:x_1^2+x_2^2=x_0x_2+x_1x_3,\text{ }\text{ }\text{ }Q_3:x_1^2-x_2^2=x_0x_2-x_1x_3.$$ 
Note that $(\tau_1,\tau_2)$ acts on the equation of:
\begin{itemize}
    \item $Q_1$ by multiplying it by $(1,-1)$,
     \item $Q_2$ by multiplying it by $(1,\text{ }1)$,
      \item $Q_3$ by multiplying it by $(-1,1)$.
\end{itemize}
Thus since all the pairs are different and $\MM$ is $3$-dimensional there are exactly $3$ $G$-invariant quadrics which we listed above. Note that these quadrics are smooth.
\end{proof}

\noindent Now take a $G$-invariant quadric $Q\in\MM$ and look at the intersection of it with $L$. Note that $L\not \subset Q$ since $L$ does not intersect $C$. $Q\cap L$ cannot be one point since we do not have $G$-fixed points thus The intersection $Q\cap L$ consists of two distinct points. This pair of points does not lie on curves of bidegree $(1,0)$ or $(0,1)$ (since these are the lines on $Q$ and we know that $L\not \subset Q$). Now we see that the blowup $\tilde{Q}\to Q$ at these points is a del Pezzo surface of degree $6$.

\subsubsection{$\DZ/2\DZ\times \DZ/2\DZ$-invariant lines}\label{3}
Let us describe $G$-invariant lines in $\DP^3$. As in proof of Lemma \ref{lemmasystemofquadrics}, we may assume that $G$ is generated by
$$\tau_1:[x:y]\to[y:x],$$
$$\tau_2:[x:y]\to[x:-y].$$
Which gives us the action on $\DP^3$ by:
$$\tau_1:[x_0:x_1:x_2:x_3]\to [x_3:x_2:x_1:x_0],$$
$$\tau_2:[x_0:x_1:x_2:x_3]\to [x_0:-x_1:x_2:-x_3].$$
In this case all $G$-invariant lines are of the form:
$$\begin{cases}\lambda x_0+\mu x_2=0,\\
\lambda x_3+\mu x_1=0,
\end{cases}$$
 where $[\lambda:\mu]\in\DP^1$. All such lines do not intersect each other  and lie on the quadric $Q_4$ given by $x_1x_0=x_2x_3$. We see that $\DP^3$ contains infinitely may $G$-invariant lines and all of them are contained in $Q_4$. Among them there are $3$ lines that intersect $C$. We can describe them explicitly. Let's now look at the intersection of this quadric with $C$. There are exactly $6$ such points:
   \begin{align*}
        P_1&=[0:1]=[0:0:0:1],\nonumber\\
        P_2&=[1:0]=[1:0:0:0],\nonumber\\
        P_3&=[1:1]=[1:1:1:1],\nonumber\\
        P_4&=[1:-1]=[1:-1:1:-1],\nonumber\\
        P_5&=[1:i]=[1:i:-1:-i],\nonumber\\
        P_6&=[1:-i]=[1:-i:-1:i],\nonumber
    \end{align*}
 here in the third column are given the corresponding coordinates on $C\subset\DP^3$. Note that $\tau_1$ exchanges $P_1$ and $P_2$, $\tau_2$ exchanges $P_3$ and $P_4$, $\tau_2$ exchanges $P_5$ and $P_6$ and we obtain three pairs of points which belong to the same line (different one for each pair). We denote these lines $L_{12}$, $L_{34}$, $L_{56}$, where $L_{ij}$ is the line connecting points $P_i$ and $P_j$.
 \begin{lemma}\label{pi(Z)not subset R} Suppose that $Z$ is an irreducible curve on $X$, $\pi(Z)$ is its image on $\DP^3$ and $\pi(Z)$ is a line different from $L$, $L_{12}$, $L_{34}$, $L_{56}$ then $Z\not \subset R$.
 \end{lemma}
 \begin{proof}
  Suppose $Z$ is contained in $R$.  Consider the following commutative diagram from section \ref{1}:
$$\xymatrix{& R\subset X\ar[dr]^{\phi}\ar[dd]_{\pi} \ar@/^2pc/[rrdd]^{\sigma}
\\&&V\ar[dl]_{\vartheta}\ar[dr]^{\xi}\\
&\pi(R)\subset\DP^3\ar@{-->}[rr]&&\DP^2\supset \xi\circ\phi(R)}$$
Where the bottom dashed arrow is given by the linear system of quadrics containing $C$. Using the equations of quadrics which form the basis of the linear system $\MM$ defined in Section \ref{2} we get the explicit map:
$$\DP^3\xdashrightarrow{} \DP^2\text{ where }[x_0:x_1:x_2:x_3]\xdashrightarrow{}[x_0x_3-x_1x_2:x_1^2-x_0x_2:x_2^2-x_1x_3].$$
We know that $\xi\circ\phi(R)$ is a conic. Let's write its equation in $\DP^2$ with coordinates $[x:y:z]$:
$$a_1x^2 + a_2xy + a_3xz + a_4y^2 + a_5yz + a_6z^2=0.$$
We want to look at the preimage of this equation in $\DP^3$ which will give the equation for $\pi(R)$. Substituting $[x_0x_3-x_1x_2:x_1^2-x_0x_2:x_2^2-x_1x_3]$ into the defining equation of $\xi\circ\phi(R)$ we get:
$$\pi(R): a_1x_3^2x_0^2 - a_2x_2x_3x_0^2 + a_4x_2^2x_0^2 + a_2x_3x_0x_1^2 - 2a_4x_2x_0x_1^2 + a_2x_2^2x_0x_1 + (-2a_1 + a_5)x_2x_3x_0x_1 -$$$$- a_3x_3^2x_0x_1 + a_3x_3x_2^2x_0 - a_5x_2^3x_0 + a_4x_1^4 - a_2x_1^3x_2 - a_5x_3x_1^3 + $$$$+(a_1 + a_5)x_2^2x_1^2 + a_3x_2x_3x_1^2 + a_6x_3^2x_1^2 - a_3x_2^3x_1 - 2a_6x_2^2x_3x_1 + a_6x_2^4=0.$$
Recall from section \ref{3} that all $G$-invariant lines are of the form
$$\begin{cases}\lambda x_0+\mu x_2=0,\\
\lambda x_3+\mu x_1=0,
\end{cases}$$
 where $[\lambda:\mu]\in\DP^1$. Now $L$ is given by
$$L=L_s:\begin{cases}x_0+s x_2=0,\\
x_3+s x_1=0,
\end{cases}$$
 for $s \in \DC$. Note that $s\ne 0$ since otherwise $X$ would have an infinite group of automorphisms. Similarly  $\pi(Z)$ is given by 
$$\pi(Z)=L_t:\begin{cases}x_0+t x_2=0,\\
x_3+t x_1=0,
\end{cases}$$
  for $t \in \DC$. By our assumption $L_s$ is contained in $\pi(R)$. This gives
$$\{a_1 = -1/s,\text{ } a_2 = 0,\text{ } a_3 = 0,\text{ } a_4 = 1, \text{ } a_5 = (s^2 + 1)/s, a_6 = 1\}.$$
So that $\pi(R)$ is given by
$$\pi(R): x_2^2x_0^2s - x_3^2x_0^2 - 2x_2x_0x_1^2s + (s^2 + 3)x_2x_3x_0x_1 + (-s^2 - 1)x_2^3x_0 + sx_1^4 +$$$$+ (-s^2 - 1)x_3x_1^3 + s^2x_1^2x_2^2 + x_3^2x_1^2s - 2x_2^2x_3x_1s + sx_2^4=0.$$
 Similarly since $L_t$ is contained in $\pi(R)$ we get
 
 $$\begin{cases}
 -t^4 - 4ts + (s^2 + 3)t^2 + s^2=0,\\
 s + (-s^2 - 1)t + t^2s=0.
 \end{cases}
 $$
 and the solution to this system is $s=t$ which means that $L_s=L$ and $L_t=\pi(Z)$ coincide  contradicting the assumption on $Z$.
 \end{proof}
 \begin{remark}
Let us use the assumptions and the notations from the proof of Lemma \ref{pi(Z)not subset R}. There is another way to show that the points in $\{b_1,b_2,b_3,b_4\}$ are in general position which as discussed above is equivalent to showing that  $Z\not\subset R$. Suppose the opposite, i.e. that $Z\subset R$. Now  as in the proof of lemma above we take suitable  $t,s\in \DC$ such that $L_s=L$ and $L_t=\pi(Z)$.   Consider the following commutative diagram:
$$\xymatrix{ &X\ar[dl]_{\pi}\ar[dr]^{\sigma}\\
\DP^3\ar@{-->}[rr]&&\DP^2}$$
Where the bottom dashed arrow is given by the linear system $\MM$ of quadrics containing $C$ and a morphism $\sigma|_{\tilde{Q}_4}$ is given by the linear system $|\pi^*(2H)-E_C|$. Now we restrict this diagram to the quadric $Q_4$ which is the quadric which consists of $G$-invariant curves, $\tilde{Q}_4$ is the strict transform of $Q_4$ on $X$ :
$$\xymatrix{ &\tilde{Q}_4\ar[dl]_{\pi|_{\tilde{Q}_4}}\ar[dr]^{\sigma|_{\tilde{Q}_4}}\\
Q_4\ar@{-->}[rr]_{\MM|_{\tilde{Q}_4}}&&\DP^2}$$
The restriction $\pi|_{\tilde{Q}_4}$ is the blow up of the intersection points $C\cap Q_4=\{P_1,...,P_6\}$,
which we described in \ref{3}. Observe that $\tilde{Q}_4$ is not a del Pezzo surface. Indeed, let $\tilde{L}_{ij}$ be in $\{\tilde{L}_{12},\tilde{L}_{34}, \tilde{L}_{56}\}$ (here $L_{ij}$ is  line connecting $P_i$ and $P_j$) then  $\tilde{L}_{12}$, $\tilde{L}_{34}$,  $\tilde{L}_{56}$ where  $\tilde{L}_{ij}$ is the strict transform of the line $L_{ij}$  are (-2)-curves in $\tilde{Q}_4$  since the lines $L_{12}$, $L_{34}$,  $L_{56}$ lie in $Q_4$.
So that they has trivial intersection with $-K_{\tilde{Q}_4}$. In fact, using the coordinates of the points $P_1,...,P_6$ and the equation of $Q_4$ one can show that
the lines $L_{ij}$ are the only secants of the curve $C$ that are contained in $Q_4$. On the other hand, $$-K_{\tilde{Q}_4}\sim (\pi^*(2H)-E_C)\vert_{\tilde{Q}_4},$$
so that $-K_{\tilde{Q}_4}$ is nef and big, and the only curves in $\tilde{Q}_4$ that has trivial intersection with $-K_{\tilde{Q}_4}$ are the three curves $\tilde{L}_{ij}$.
Note that these curves do not intersect the curve $Z$.
 \\ Taking the Stein factorization of the morphism $\sigma\vert_{\tilde{Q}_4}: \tilde{Q}_4\to \DP^2$, we obtain the following commutative diagram:
  $$\xymatrix{\tilde{Q}_4\ar[d]_{\pi|_{\tilde{Q}_4}}\ar[rrrr]^{\text{contraction of (-2) curves}}\ar[drrrr]^{\sigma\vert_{\tilde{Q}_4}} &&&&\bar{Q}_4\ar[d]^{\beta}\\
  Q_4\ar@{-->}[rrrr]&&&&\DP^2}$$
  where $\bar{Q}_4$ is a singular del Pezzo surface of degree $2$ with three singular points of type $A_1$,
$\beta$ is the double cover given by $|-K_{\tilde{Q}_4}|$, and the dasharrow is the rational map given by the restriction of the linear system $\mathcal{M}$ to $Q_4$. Suppose $\bar{L}_s$, $\bar{L}_t$ are the images of  $\tilde{L}_s$, $\tilde{L}_t$ which are the strict transforms of $L_s$ and $L_t$ on $\tilde{Q}_4$.  $\bar{L}_s$ and $\bar{L}_t$ do not pass through singular points, because $L_s$ and $L_t$ are disjoint from the lines $L_{ij}$. Since both $L_t,L_s\subset \pi(R)$ by assumption then by construction of $R$ given in Section \ref{1} we get that $\beta(\bar{L}_s)=\beta(\bar{L}_t)$ is the same conic in $\DP^2$. So we see that:
  $$\bar{L}_s+\bar{L}_t=\beta^*(\MO_{\DP^2}(2))=-2K_{\bar{Q}_4}.$$
  By the adjunction formula we have:
  $$K_{\bar{Q}_4}\cdot \bar{L}_s+\bar{L}_s^2=-2\Rightarrow -K_{\bar{Q}_4}\cdot \bar{L}_s=2.$$
  So we get 
  $$0=\bar{L}_s\cdot \bar{L}_s+\bar{L}_t\cdot \bar{L}_s=-2K_{\bar{Q}_4}\cdot \bar{L}_s=4.$$
  Which gives us a contradiction. Thus $Z\not\subset R $.
 \end{remark}
\subsection{Mori Cone}
  Let $l_L$, $l_C$, $l_R$ be the general fibers of the natural projections $E_L\to L$, $E_C\to C$, $R\to \sigma(R)$.  Observe that we can contract any of two rays $\DR_{\ge 0}[l_L]$, $\DR_{\ge 0}[l_C]$, $\DR_{\ge 0}[l_R]$. Indeed $l_C$ and $l_L$ are contracted by $\pi:X\to \DP^3$, $l_R$ and $l_L$ are contracted by $\sigma:X\to \DP^2$, $l_R$ and $l_C$ are contracted by $\eta:X\to \DP^1$. Thus, these curves generate $3$ extreme rays $\DR_{\ge 0}[l_L]$, $\DR_{\ge 0}[l_C]$, $\DR_{\ge 0}[l_R]$ of the Mori cone $\overline{\mathrm{NE}(X)}$.
 \subsection{Cone of Effective Divisors}
\begin{lemma}
Suppose $S$ is a surface in $X$ then
$$S\sim a(\pi^*(H)-E_L)+b(2\pi^*(H)-E_C)+cR+eE_L+fE_C,$$
for $a,b,c,e,f\in \DZ_{\ge 0}$.
\end{lemma}
\begin{proof}
Suppose $\pi(S)\subset \DP^3$ is the surface of degree $d$ in $\DP^3$. Then we have
$$S\sim d\pi^*(H)-m_LE_L-m_CE_C,$$
where $m_L$ is the multiplicity of $\pi(S)$ in $L$, $m_C$ is the multiplicity of $\pi(S)$ in $C$. Suppose that $S\ne E_C$, $S\ne E_L$ and $S\ne R$ for all $n$. Now let's intersect $S$ with three extreme rays $l_L$, $l_C$, $l_R$ corresponding to $L$, $C$, $R$:
\\\\\begin{minipage}{0.3\textwidth}
\begin{itemize}
    \item $\pi^*(H)\cdot l_C=0$,
    \item $\pi^*(H)\cdot l_L=0$,
    \item $\pi^*(H)\cdot l_R=1$,
\end{itemize}
\end{minipage}
\hfill
\begin{minipage}{0.3\textwidth}
\begin{itemize}
    \item $E_L\cdot l_C=0$,
    \item $E_L\cdot l_L=-1$,
    \item $E_L\cdot l_R=1$,
\end{itemize}
\end{minipage}
\hfill
\begin{minipage}{0.3\textwidth}
\begin{itemize}
    \item $E_C\cdot l_C=-1$,
    \item $E_C\cdot l_L=0$,
    \item $E_C\cdot l_R=2$.
\end{itemize}
\end{minipage}
\\So we have that:
$$S\cdot l_C=m_C\ge 0,\text{ }\text{ }\text{ }S\cdot l_L=m_L\ge 0,\text{ }\text{ }\text{ }S\cdot l_R=d-m_L-2m_C\ge 0.$$
Moreover if $l_1$ is the general line intersecting $L$, $l_2$ is the general secant of $C$ then we get strict inequalities:
$$S\cdot l_1=d-m_L>0,\text{ }\text{ }\text{ }S\cdot l_2=d-2m_C>0$$
Now we want to find the integer positive solutions for:
$$d\pi^*(H)-m_LE_L-m_CE_C=a(\pi^*(H)-E_L)+b(2\pi^*(H)-E_C)+cR+eE_L+fE_C.$$
Comparing the coefficients we get:
$$d = a + 2b + 4c,\text{ }\text{ }\text{ } m_C = b + 2c - f,\text{ }\text{ }\text{ } m_L = a + c - e.$$
The non-negative solution to this system can be given by
$$\begin{cases}
a = -2m_C + d,\\
b = m_C,\\
c = 0,\\
e = -m_L - 2m_C + d,\\
f = 0.
\end{cases}$$
Thus, the cone of effective divisors over $\DZ$ is generated by $\pi^*(H)-E_L$, $2\pi^*(H)-E_C$, $R$, $E_L$, $E_C$.
\end{proof}

\begin{corollary}\label{4}
Cone of effective divisors of $X$ is generated over $\DQ$ by $\pi^*(H)-E_L$, $R$, $E_L$, $E_C$. More precisely,  suppose $S$ is a surface in $X$ then
$$S\sim_{\DQ} a(\pi^*(H)-E_L)+cR+eE_L+fE_C,$$
for unique $a,c,e,f\in \DQ_{\ge 0}$.
\end{corollary}
 
\section{Proof of the Main Theorem}
\noindent Let  $C$ be a twisted cubic in $\DP^3$ that is the image of the map $\DP^1\hookrightarrow \DP^3$ given by
$$[x:y]\to[x^3:x^2y:xy^2:y^3],$$
and $L$ be a line in $\DP^3$ that is disjoint from $C$ given by 
$$L:\begin{cases}x_0=rx_1,\\
x_3=rx_2
\end{cases}$$
where $r\ne 0$, $r\ne \pm 1$, $r\ne \pm 3$ in the coordinates presented in section \ref{types}.  $X$ is a smooth Fano 3-folds in Family 3.12 obtained by blowing up $\DP^3$ in $C$ and $L$, and $G$ is the subgroup in $\Aut(X)$ such that $G\cong \DZ/2\DZ\times\DZ/2\DZ$ described in Section \ref{6}. \\Suppose $X$ is not $K$-polystable.
By \mbox{Theorem~\ref{theorem:G-K-polystability}}, there exists a~$G$-invariant prime divisor $F$ over $X$
such that $\beta(F)=A_X(F)-S_X(F)\leqslant 0$. Let us seek for a~contradiction.
Let $Z=C_X(F)$. Then $Z$ is not a~point by Corollary~\ref{fixedpoints}, and $Z$ is not a~surface by Theorem~\ref{theorem:divisorially-stable},
so that $Z$ is a~$G$-invariant irreducible curve.

\begin{lemma}\label{lemma:3-12-Z-lines}
Suppose that  $\pi(Z)\ne L$ then $\pi(Z)$ is not one of the $G$-invariant lines which does not intersect $C$.
\end{lemma}
\begin{proof}
Let's take a $G$-invariant line $\pi(Z)$ that does not intersect $C$ and consider a plane $H$ which contains this line. It intersects a line $L$ in one point and a twisted cubic $C$ at three points. Let $S$ be the proper transform of $H$ on $X$. In this case we have that the induced map $\pi|_S:S\to H$ is the blowup of a plane $H$ in $4$ points $b_1=H\cap L$, $b_2,b_3,b_4=H\cap C$. We now need to check that these points are in general position to conclude that $S$ is a del Pezzo surface of degree $5$.

To prove that  we need to show that the points in $\{b_1,b_2,b_3,b_4\}$ are in general position which means that no three of them belong to the same line. Note that $b_2,b_3,b_4=H\cap C$ does not belong to the same line, because $C$ is an intersection of quadrics. So the only option is that $b_1$ and two points from the set $\{b_2,b_3,b_4\}$ belongs to the same line. Suppose $H$ is a general plane and $b_1$ and $2$ points among $\{b_2, b_3, b_4\}$ are contained in one line $\ell$. From Section \ref{1} we know that $\pi(R)$ is spanned by secants of $C$ that intersect $L$, so $H$ contains such secant $\ell$. Moreover $\pi(Z)$ intersects $\ell$, so we see that $\pi(Z)$ intersects a general secant of $C$ that is contained in $\pi(R)$. Then  $Z\subset R$ which contradicts Lemma \ref{pi(Z)not subset R}.   So we can choose the hyperplane $H$ in such a way that the points in $\{b_1,b_2,b_3,b_4\}$ are in general position.
\\Thus,  $S$ is a del Pezzo surface of degree $5$ with the exceptional divisors $E_1$, $E_2$, $E_3$, $E_4$ corresponding to points $b_1,b_2,b_3,b_4$, $L_{ij}$ are the preimages of lines connecting $b_i$ and $b_j$ for $i\in \{1,...,4\}$. Recall that $E_1$, $E_2$, $E_3$, $E_4$ and $L_{ij}$ generate the Mori Cone $\overline{\mathrm{NE}(S)}$.   We have that 
 $$-K_X\sim \pi^*(4H)-E_C-E_L,$$
 $$R\sim \pi^*(4H)-2E_C-E_L,$$
and moreover
$$\pi^*(H)|_S\sim S|_S\sim Z,\text{ }\text{ }\text{ }E_L|_S\sim E_1,\text{ }\text{ }\text{ }E_C|_S\sim E_2+E_3+E_4.$$  
\noindent 
By Theorem~\ref{theorem:divisorially-stable}, we have~\mbox{$S_X(S)<1$}.
Thus, we conclude that $S(W_{\bullet,\bullet}^{S};Z)\geqslant 1$ by  Corollary~\ref{corollary:Kento-formula-Fano-threefold-surface-curve}.
Let us compute $S(W_{\bullet,\bullet}^{S};Z)$.
Take $u\in\mathbb{R}_{\geqslant 0}$. Observe that $$-K_X-uS\sim(1-u/3)R+u/3(H-E_L)+(1-2u/3)E_C,$$
which implies that $-K_X-uS$ is pseudo-effective if and only if $u\leqslant\frac{3}{2}$.
Let $P(u)=P(-K_X-uS)$ be a positive part of Zariski decomposition and $N(u)=N(-K_X-uS)$  be a negative part of Zariski decomposition. Here we use the notations introduced in Theorem \ref{corollary:Kento-formula-Fano-threefold-surface-curve}.
$$
P(u)=
\left\{\aligned
&-K_X-uS\ \text{if $0\leqslant u\leqslant 1$}, \\
&-K_X-uS-(u-1)R\ \text{if $1\leqslant u\leqslant\frac{3}{2}$},
\endaligned
\right.
$$
and $$
N(u)=
\left\{\aligned
&0\ \text{if $0\leqslant u\leqslant 1$}, \\
&(u-1)R\ \text{if $1\leqslant u\leqslant\frac{3}{2}$}.
\endaligned
\right.
$$
Then take any $v\in\mathbb{R}_{\geqslant 0}$. Suppose $P(u,v)$ is a positive part of the Zariski decomposition of $(-K_X-uS)|_S-vZ$,  $N(u,v)$ is a negative part of the Zariski decomposition of $(-K_X-uS)|_S-vZ$.
\\If $u\in[0,1]$, we have $$P(u)\vert_{S}-vZ\sim (-K_X-uS)|_S-vZ= (4H-E_C-E_L-uS)|_S-vZ=$$
$$= 4Z-(E_2+E_3+E_4)-E_1-uZ-vZ=(4-u-v)Z-(E_1+E_2+E_3+E_4).$$
For $u\in[0,1]$ we find $v$ such that the divisor $P(u)\vert_{S}-vZ$ is nef. We have that:
\begin{itemize}
    \item $P(u,v)\cdot Z=4-u-v$,
\item $P(u,v)\cdot E_i=1,\text{ for }i\in\{1,...,4\}$,
\item $P(u,v)\cdot L_{ij}=2-u-v,\text{ for }i,j\in\{1,...,4\}$.
\end{itemize}
To check when the divisor is nef we should choose the strongest inequality from the following system:
$$\begin{cases}
4-u-v\ge 0,\\
2-u-v \ge 0,
\end{cases}\Rightarrow v\le 2-u.$$
Thus for $v\le 2-u$ we have that the divisor $P(u)|_S-vZ$ is nef. Note that for $v=2-u$ we have that $P(u,2-u)^2=0$ so $P(u)|_S-vZ$ is pseudo-effective until $v$ satisfies $v\le 2-u$. We see that the Zariski decomposition for $u\in[0,1]$, $v\le 2-u$ is given by
$$P(u,v)=(4-u-v)Z-(E_1+E_2+E_3+E_4),\text{ }\text{ }\text{ }N(u,v)=0.$$
\\For $u\in[1,3/2]$  and $v\in\mathbb{R}_{\geqslant 0}$ we have that:
$$P(u)\vert_{S}-vZ\sim (-K_X-uS-(u-1)R)|_S-vZ=$$
$$=(4H-E_C-E_L-uS-(u-1)(4H-2E_C-E_L))|_S-vZ=$$
$$=(8-5u-v)Z+(2u-3)(E_2+E_3+E_4)+(u-2)E_1.$$
We have that:
\begin{itemize}
\item $P(u,v)\cdot Z=8-5u-v$,
\item $P(u,v)\cdot E_1=2-u$,
\item $P(u,v)\cdot E_i=3-2u$ for $i\in\{2,3,4\}$,
\item $P(u,v)\cdot L_{1i}=3-2u-v$ for $L_{1i}\in\{L_{12},L_{13},L_{14}\}$,
\item $P(u,v)\cdot L_{ij}=3-2u-v$ for $L_{ij}\in\{L_{23},L_{34},L_{24}\}$.
\end{itemize}
To check when the divisor is nef we should choose the strongest inequality from the following system:
$$\begin{cases}
8-5u-v\ge 0,\\
3-2u-v\ge0,\\
2-u-v\ge 0,
\end{cases}\Rightarrow v\le 3-2u.$$
Thus, for $v\le 3-2u$ we have that the divisor $P(u)|_S-vZ$ is nef. We see that the Zariski decomposition for $u\in[1,3/2]$ $v\le 3-2u$ is given by
$$P(u,v)=(8-5u-v)Z+(2u-3)(E_2+E_3+E_4)+(u-2)E_1,\text{ }\text{ }\text{ }N(u,v)=0.$$
Suppose $v\ge 3-2u$. Let us find the Zariski decomposition of $P(u)|_S-vZ$.  $P(u,v)$ is given by
$$P(u,v)=(8-5u-v)Z+(2u-3)(E_2+E_3+E_4)+(u-2)E_1-aL_{12}-bL_{13}-cL_{14},$$
for some $a,b,c$. Note that we should have $P(u,v)\cdot L_{1i}=0$ for $i\in\{2,3,4\}$ thus $a=b=c =-(3-2u-v)$. So we obtain that:
$$P(u,v)=(8-5u-v)Z+(2u-3)(E_2+E_3+E_4)+(u-2)E_1+(3-2u-v)(L_{12}+L_{13}+L_{14}),$$
$$N(u,v)=-(3-2u-v)(L_{12}+L_{13}+L_{14}).$$
We have that:
\begin{itemize}
    \item $P(u,v)\cdot Z=-11u-4v+17$,
\item $P(u,v)\cdot E_1=-7u-3v+11$,
\item $P(u,v)\cdot E_i=-v-4u+6$ for $i\in\{2,3,4\}$,
\item $P(u,v)\cdot L_{1i}=0$ for $L_{1i}\in\{L_{12},L_{13},L_{14}\}$,
\item $P(u,v)\cdot L_{ij}=5-3u-2v$ for $L_{ij}\in\{L_{23},L_{24},L_{34}\}$.
\end{itemize}
To check when the divisor is pseudo-effective we should choose the strongest inequality from the following system:
$$\begin{cases}-11u-4v+17\ge 0,
\\-7u-3v+11\ge 0,\\
-v-4u+6\ge 0,\\
5-3u-2v\ge 0,
\end{cases}\Rightarrow
\begin{cases}
v\le \frac{-11u+17}{4},\\
v\le \frac{-7u+11}{3},\\
v\le 6-4u,\\
v\le \frac{-3u+5}{2}.
\end{cases}$$
So for $u\in[1,7/5]$ we get $v\le \frac{-3u+5}{2}$ and for $u\in[7/5,3/2]$ we get $v\le 6-4u$. Note that $$P(u,v)^2=2(3u - 5 + 2v)(4u - 6 + v).$$ Thus, for $v= \frac{-3u+5}{2}$ or $v = 6-4u$ we have that $P(u,v)^2=0$. Note that for $v=2-u$ we have that $P(u,v)^2=0$ so $P(u)|_S-vZ$ is pseudo-effective until $v$ is in these intervals.
The Corollary~\ref{corollary:Kento-formula-Fano-threefold-surface-curve} gives us
$$\hspace*{-0.5cm}
S\big(W^S_{\bullet,\bullet};Z\big)=\frac{3}{(-K_X)^3}\int_0^{3/2}\big(P(u)^{2}\cdot S\big)\cdot\mathrm{ord}_{Z}\Big(N(u)\big\vert_{S}\Big)du+\frac{3}{(-K_X)^3}\int_0^{3/2}\int_0^\infty \mathrm{vol}\big(P(u)\big\vert_{S}-vZ\big)dvdu.
$$
Note that $\mathrm{ord}_{Z}\Big(N(u)\big\vert_{S}\Big)=0$ because $Z\not\subset R$. So we are only left with the second part of the integral which equals:
$$S\big(W^S_{\bullet,\bullet};Z\big)=\frac{3}{28}\int_0^1\int_0^{2-u}\big((4-u-v)Z-(E_1+E_2+E_3+E_4)\big)^2dvdu+$$
$$+\frac{3}{28}\int_1^{3/2}\int_0^{2u-3}\big((8-5u-v)Z+(2u-3)(E_2+E_3+E_4)+(u-2)E_1\big)^2dvdu+$$
$$\hspace*{-0.5cm}+\frac{3}{28}\int_1^{7/5}\int_{3-2u}^{\frac{5-3u}{2}}\big((8-5u-v)Z+(2u-3)(E_2+E_3+E_4)+(u-2)E_1)+(3-2u-v)(L_{12}+L_{13}+L_{14})\big)^2dvdu+$$
$$\hspace*{-0.5cm}+\frac{3}{28}\int_{7/5}^{3/2}\int_{3-2u}^{6-4u}\big((8-5u-v)Z+(2u-3)(E_2+E_3+E_4)+(u-2)E_1)+(3-2u-v)(L_{12}+L_{13}+L_{14})\big)^2dvdu=$$
$$=\frac{3}{28}\int_0^1\int_0^{2-u}\big((4-u-v)^2-4\big)dvdu+\frac{3}{28}\int_1^{3/2}\int_0^{2u-3}\big(12u^2 + 10uv + v^2 - 40u - 16v + 33\big)dvdu+$$
$$+\frac{3}{28}\int_1^{7/5}\int_{3-2u}^{\frac{5-3u}{2}}\big(24u^2 + 22uv + 4v^2 - 76u - 34v + 60\big)dvdu+$$
$$+\frac{3}{28}\int_1^{7/5}\int_{3-2u}^{6-4u}\big(24u^2 + 22uv + 4v^2 - 76u - 34v + 60\big)dvdu=\frac{753}{1120}<1.$$
The obtained contradiction completes the~proof of the~lemma.
\end{proof}

\begin{lemma}
\label{lemma:3-12-Z-in-E}
One has $Z\not\subset E_L$.
\end{lemma}

\begin{proof}
Suppose that $Z\subset E_L$. Observe that $E_L\cong\mathbb{P}^1\times\mathbb{P}^1$.
Let $\mathbf{s}$ be the~section of the~natural projection $E_L\to L$ such that $\mathbf{s}^2=0$, and $\mathbf{l}$ be a~fiber of this projection.
Then
$$E_L\vert_{E_L}\sim -\mathbf{s}+\mathbf{l}\text{, }\pi^*(H)\vert_{E_L}\sim\mathbf{l}\text{, }R\vert_{E_L}\sim \mathbf{s}+3\mathbf{l},$$
and  $E_C$ and $E_L$ are disjoint. By Theorem~\ref{theorem:divisorially-stable}, we have~\mbox{$S_X(E_L)<1$}.
Thus, we conclude that $S(W_{\bullet,\bullet}^{E_L};Z)\geqslant 1$ by  Corollary~\ref{corollary:Kento-formula-Fano-threefold-surface-curve}.
Let us compute $S(W_{\bullet,\bullet}^{E_L};Z)$.
Take $u\in\mathbb{R}_{\geqslant 0}$. Observe that
$$-K_X-uE_L\sim_{\mathbb{R}}\frac{1}{2}R+2(\pi^*(H)-E_L)+\big(\frac{3}{2}-u\big)E_L,$$
which implies that $-K_X-uE_L$ is pseudo-effective if and only if $u\leqslant\frac{3}{2}$.
Let $P(u)=P(-K_X-uE_L)$ and $N(u)=N(-K_X-uE_L)$. Then
$$
P(u)=
\left\{\aligned
&-K_X-uE_L\ \text{for $0\leqslant u\leqslant 1$}, \\
&(8-4u)\pi^*(H)-(3-2u)E_C-2E_L\ \text{for $1\leqslant u\leqslant\frac{3}{2}$},
\endaligned
\right.
$$
and $$
N(u)=
\left\{\aligned
&0\ \text{for $0\leqslant u\leqslant 1$}, \\
&(u-1)R\ \text{for $1\leqslant u\leqslant\frac{3}{2}$}.
\endaligned
\right.
$$
\\\\$\boxed{1}$ Suppose that $Z\ne R\vert_{E_L}$ and $Z\sim a\mathbf{s}+b \mathbf{l}$. Note that $a\ge 1$ since $\DP^3$ does not contain $G$-fixed points. Then using Corollary~\ref{corollary:Kento-formula-Fano-threefold-surface-curve} we obtain
$$
S\big(W_{\bullet,\bullet}^{E_L};Z\big)=\frac{3}{28}\int_{0}^{\frac{3}{2}}\int_{0}^{\infty}\mathrm{vol}\Big(P(u)\big\vert_{Q}-v(a\mathbf{s}+b \mathbf{l})\Big)dvdu\leqslant \frac{3}{28}\int_{0}^{\frac{3}{2}}\int_{0}^{\infty}\mathrm{vol}\Big(P(u)\big\vert_{Q}-v\mathbf{s}\Big)dvdu.
$$
so it is enough to show that the last integral is less than $1$ to deduce a contradiction. So suppose $Z\sim \mathbf{s}$. We have that:
$$P(u)\vert_{E_L}-v \mathbf{s}\sim_{\mathbb{R}}\begin{cases}(1 + u - v)\mathbf{s} + (3 - u)\mathbf{l}\text{ for }0\le u\le 1,\\
(2 - v)\mathbf{s} + (6 - 4u)\mathbf{l}\text{ for }1\le u\le 3/2.
\end{cases}$$
Then  Corollary~\ref{corollary:Kento-formula-Fano-threefold-surface-curve} gives
$$S\big(W_{\bullet,\bullet}^{E_L}; \mathbf{s}\big)=\frac{3}{28}\int_{0}^{\frac{3}{2}}\int_{0}^{\infty}\mathrm{vol}\Big(P(u)\big\vert_{Q}-v\mathbf{s}\Big)dvdu=$$
$$\frac{3}{28}\int_{0}^{1}\int_{0}^{1+u}2(3 - u)(1 + u - v)dvdu+\frac{3}{28}\int_{1}^{3/2}\int_{0}^{2}4(v-2 )(-3 + 2u)dvdu=\frac{13}{16}<1.$$
So we obtained a contradiction with Corollary~\ref{corollary:Kento-formula-Fano-threefold-surface-curve}.
\\\\ Thus, for any $G$-invariant curve $Z\subset E_L$ such that $Z\ne R\vert_{E_L}$ we have $S\big(W_{\bullet,\bullet}^{Q};Z\big)<1$-contradiction with Corollary~\ref{corollary:Kento-formula-Fano-threefold-surface-curve}.
\\\\$\boxed{2}$ Suppose $Z=R\vert_{E_L}\sim \mathbf{s}+3\mathbf{l}$.
\noindent
Take any $v\in\mathbb{R}_{\geqslant 0}$ then we have: $$P(u)\vert_{E_L}-vZ\sim_{\mathbb{R}}\begin{cases}(1+u-v)\mathbf{s}+(3-u-3v)\mathbf{l},\text{ for }0\le u \le 1,\\
(2-v)\mathbf{s}+(6-4u-3v)\mathbf{l}, \text{ for }1\le u \le 3/2.
\end{cases}$$
Hence, if $Z=R\vert_{E_L}$, then  Corollary~\ref{corollary:Kento-formula-Fano-threefold-surface-curve} gives
\begin{multline*}
S\big(W_{\bullet,\bullet}^{E_L};Z\big)=\frac{3}{28}\int_{1}^{\frac{3}{2}}(u-1)E_L\cdot\big((8-4u)\pi^*(H)-(3-2u)E_C-2E_L\big)^2du+\\
+\frac{3}{28}\int_{0}^{\frac{3}{2}}\int_{0}^{\infty}\mathrm{vol}\Big(P(u)\big\vert_{E_L}-vZ\Big)dvdu=\\
=\frac{3}{28}\int_{1}^{\frac{3}{2}}4(u-1)(6-4u)du+\frac{3}{28}\int_{0}^{1}\int_{0}^{\frac{3-u}{3}}2(1+u-v)(3-u-3v)dvdu+\\
+\frac{3}{28}\int_{1}^{\frac{3}{2}}\int_{0}^{\frac{6-4u}{3}}2(2-v)(6-4u-3v)dvdu=\frac{19}{56}<1.\quad\quad\quad\quad\quad\quad
\end{multline*}
The obtained contradiction with Corollary~\ref{corollary:Kento-formula-Fano-threefold-surface-curve} and completes the~proof of the~lemma.
\end{proof}

\begin{lemma}
\label{lemma:3-12-Z-in-Q}
Let $Q$ be a $G$-invariant quadric surface in $\DP^3$ that contains $C$. Then $Z\not\subset Q$.
\end{lemma}

\begin{proof}
Suppose that $Z\subset Q$. Let us seek for a~contradiction.
Recall that $\pi(Q)$ is a~smooth quadric surface in $\mathbb{P}^3$ that  contains the~twisted cubic curve $C$, and it does not contain line $L$.
Let us identify $\pi(Q)=\mathbb{P}^1\times\mathbb{P}^1$ such that $C$ is a~curve in $\pi(Q)$ of degree $(1,2)$.
Then $\pi$ induces a~birational morphism $\varphi\colon Q\to \mathbb{P}^1\times\mathbb{P}^1$ that is a~blow up of two intersection points $\pi(Q)\cap L$, which are not contained in the~curve~$C$.
Moreover, the~surface $Q$ is a~smooth del Pezzo surface of degree $6$,
because the~points of the~intersection $\pi(Q)\cap L$ are not contained in one line in $\pi(Q)$ since otherwise this line would be $L$. But $L$ is not contained in $\pi(Q)$ which is a contradiction.
\\
By Theorem~\ref{theorem:divisorially-stable}, we have \mbox{$S_X(Q)<1$}.
Then $S(W_{\bullet,\bullet}^{Q};Z)\geqslant 1$ by  Corollary~\ref{corollary:Kento-formula-Fano-threefold-surface-curve}.
Let us show that $S(W_{\bullet,\bullet}^{Q};Z)<1$, which would give us the~desired contradiction.
\\
Take $u\in\mathbb{R}_{\geqslant 0}$. Then $$-K_X-uQ\sim_{\mathbb{R}} 2\pi^*(H)-E_L+(1-u)(2\pi^*(H)-E_C),$$
which implies that $-K_X-uQ$ is nef for every $u\in[0,1]$.
On the~other hand, we have
$$
-K_X-uQ\sim_{\mathbb{R}} (4-2u)\big(\pi^*(H)-E_L\big)+(3-2u)E_L+(u-1)E_C,
$$
so that the~divisor $-K_X-uS$ is pseudo-effective $\iff$ $u\in[0,\frac{3}{2}]$.
 Let $P(u)=P(-K_X-uQ)$ and $N(u)=N(-K_X-uQ)$. Then  we have
$$
P\big(-K_X-uQ\big)=
\left\{\aligned
&-K_X-uQ\ \text{for $0\leqslant u\leqslant 1$}, \\
&(4-2u)\pi^*(H)-E_L\ \text{for $1\leqslant u\leqslant\frac{3}{2}$},
\endaligned
\right.
$$
and $$N(-K_X-uQ)=\begin{cases}0\text{ for }0\leqslant u\leqslant 1,\\(u-1)E_C\text{ for }1\leqslant u\leqslant\frac{3}{2}.
\end{cases}$$
\noindent
Let us introduce some notation on $Q$. Suppose $\varphi:Q\to \DP^1\times\DP^1$ is the blowup at points $A_1$, $A_2$
First, we denote by $\ell_1$ and $\ell_2$ the~proper transforms on $Q$ of general curves $k_1$ and $k_2$ in  $\mathbb{P}^1\times\mathbb{P}^1$ of degrees $(1,0)$ and $(0,1)$, respectively.
Second, we denote by $e_1$ and $e_1$ the~exceptional curves of $\varphi$ which correspond to points $A_1$, $A_2$ respectively.
Third, we let $F_{11}$, $F_{12}$, $F_{21}$, $F_{22}$ be the~$(-1)$-curves on $Q$  such that
$$F_{11}\sim \ell_1-e_1,\text{  } F_{12}\sim \ell_1-e_2,\text{  } F_{21}\sim \ell_2-e_1,\text{  } F_{22}\sim \ell_2-e_2.$$
Then $$\pi^*(H)\vert_{Q}\sim \ell_1+\ell_2,\text{  }E_L\vert_{Q}\sim e_1+e_2,\text{  }E_C\vert_{Q}\sim \ell_1+2\ell_2.$$
$\boxed{1}$ Suppose $Z\ne E_C\vert_{Q}$, then $\varphi(Z)$ is a curve since since $Z\ne e_1$ and $Z\ne e_2$, because neither $e_1$ nor $e_2$ is $G$-invariant. Now we have that $\varphi(Z)\sim ak_1+bk_2$ and so 
$$Z\sim a\ell_1+b\ell_2-m_1e_1-m_2e_2,$$
where $m_1$ is a multiplicity of $\varphi(Z)$ at point $A_1$, $m_2$ is a multiplicity of $\varphi(Z)$ at point $A_2$. Note that $G$ exchanges $A_1$ and $A_2$ and $Z$ is a $G$-invariant curve thus $m_1=m_2=:m$. We know that $Z\not \in\{F_{11}, F_{12}, F_{21}, F_{22}\}$ since $F_{ij}$-s ($i,j\in \{1,2\}$) are not $G$-invariant. Thus:
$$0\le Z\cdot F_{11}=(a\ell_1+b\ell_2-m_1e_1-m_2e_2)(l_1-e_1)=b-m\Rightarrow b\ge m,$$
$$0\le Z\cdot F_{22}=(a\ell_1+b\ell_2-m_1e_1-m_2e_2)(l_2-e_2)=a-m\Rightarrow a\ge m.$$
Now we have that 
$$Z\sim a\ell_1+b\ell_2-m(e_1+e_2)=$$
$$=\begin{cases}
\underbrace{(a-b)}_{\ge 0}\ell_1+\underbrace{m}_{\ge 0}(\ell_1+\ell_2-e_1-e_2)+\underbrace{(b-m)}_{\ge 0}(\ell_1+\ell_2),\text{ for }a\ge b,\\
\underbrace{(b-a)}_{\ge 0}\ell_2+\underbrace{m}_{\ge 0}(\ell_1+\ell_2-e_1-e_2)+\underbrace{(a-m)}_{\ge 0}(\ell_1+\ell_2),\text{ for }b\ge a.
\end{cases}$$
So we can decompose each curve $Z$ as the sum of $\ell_1$, $\ell_2$, $\ell_1+\ell_2-e_1-e_2$ with non-negative coefficients, i.e. $Z\sim c_1\ell_1+c_2\ell_2+c_3(\ell_1+\ell_2-e_1-e_2)$ for some non-negative integers $c_1$, $c_2$, $c_3$. Note if for example $c_1\ge 1$ then:
$$
S\big(W_{\bullet,\bullet}^{Q};Z\big)=\frac{3}{28}\int_{0}^{\frac{3}{2}}\int_{0}^{\infty}\mathrm{vol}\Big(P(u)\big\vert_{Q}-v(c_1\ell_1+c_2\ell_2+c_3(\ell_1+\ell_2-e_1-e_2))\Big)dvdu\leqslant$$$$\leqslant \frac{3}{28}\int_{0}^{\frac{3}{2}}\int_{0}^{\infty}\mathrm{vol}\Big(P(u)\big\vert_{Q}-v\ell_1\Big)dvdu.
$$
and similarly for $c_2$ and $c_3$. So it is enough to get $S\big(W_{\bullet,\bullet}^{Q};Z\big)<1$ for $Z\sim \ell_1$, $Z\sim \ell_2$, $Z\sim \ell_1+\ell_2-e_1-e_2$ to deduce a contradiction.
\\1). Suppose $Z\sim \ell_1$.
\\Take any $v\in \DR_{\ge 0}$. Suppose $P(u,v)$ is a positive part of the Zariski decomposition of $(-K_X-uS)|_S-vZ$,  $N(u,v)$ is a negative part of the Zariski decomposition of $(-K_X-uS)|_S-vZ$.
\\If $u\in[0,1]$ then
$$P(u)\vert_{Q}-vZ\sim (-K_X-uQ)|_Q-vZ=-e_1 - e_2 + (3 - u - v)\ell_1 + 2\ell_2.$$
Now  we find $v$ such that the divisor $P(u)\vert_{Q}-vZ$ is nef. We have that:
\\\\
\begin{minipage}{0.35\textwidth}
\begin{itemize}
    \item $P(u,v)\cdot e_1=1$,
    \item $P(u,v)\cdot e_2=1$,
    \item $P(u,v)\cdot F_{11}=1$,
\end{itemize}
\end{minipage}
\hfill
\begin{minipage}{0.8\textwidth}
\begin{itemize}
    \item $P(u,v)\cdot F_{12}=1$,
    \item $P(u,v)\cdot F_{21}=2-u-v$,
    \item $P(u,v)\cdot F_{22}=2-u-v$.
\end{itemize}
\end{minipage}
\\\\
Thus, for $v\le 2-u$ we have that the divisor $P(u)|_Q-vZ$ is nef. 
\\Suppose $v\ge 2-u$. We find the Zariski decomposition of $P(u)|_Q-vZ$. 
$$P(u,v)=(-e_1 - e_2 + (3 - u - v)\ell_1 + 2\ell_2)+(2-u-v)(F_{21}+F_{22})$$
$$=(-3 + u + v)e_1 + (-3 + u + v)e_2 + (3 - u - v)\ell_1 + (6 - 2u - 2v)\ell_2,$$
$$N(u,v)=-(2-u-v)(F_{21}+F_{22})=-(2-u-v)(2\ell_2-e_1-e_2).$$
We have that:
\\\\
\begin{minipage}{0.35\textwidth}
\begin{itemize}
    \item $P(u,v)\cdot e_1=3 - u - v$,
    \item $P(u,v)\cdot e_2=3 - u - v$,
    \item $P(u,v)\cdot F_{11}=3 - u - v$,

\end{itemize}
\end{minipage}
\hfill
\begin{minipage}{0.8\textwidth}
\begin{itemize}
    \item $P(u,v)\cdot F_{12}=3 - u - v$,
    \item $P(u,v)\cdot F_{21}=0$,
    \item $P(u,v)\cdot F_{22}=0$.
\end{itemize}
\end{minipage}
\\\\
We get $v\le 3-u$. Note that for $v=3-u$ we have that $P_2(u,3-u)^2=0$ so $P(u)|_S-vZ$ is pseudo-effective if and only if $v$ satisfies $v\le 3-u$.
\\For $u\in[1,3/2]$  and $v\in\mathbb{R}_{\geqslant 0}$ we have that:
$$P(u)\vert_{Q}-vZ\sim (-K_X-uQ)|_Q-vZ=-e_1 - e_2 + (4 - 2u - v)\ell_1 + (4 - 2u)\ell_2.$$
Now  we find $v$ such that the divisor $P(u)\vert_{Q}-vZ$ is nef. We have that:
\\\\
\begin{minipage}{0.35\textwidth}
\begin{itemize}
    \item $P(u,v)\cdot e_1=1$,
    \item $P(u,v)\cdot e_2=1$,
    \item $P(u,v)\cdot F_{11}=3 - 2u$,
\end{itemize}
\end{minipage}
\hfill
\begin{minipage}{0.8\textwidth}
\begin{itemize}
    \item $P(u,v)\cdot F_{12}=3 - 2u$,
    \item $P(u,v)\cdot F_{21}=3 - 2u - v$,
    \item $P(u,v)\cdot F_{22}=3 - 2u - v$.
\end{itemize}
\end{minipage}
\\\\
Thus for $v\le 3-2u$ we have that the divisor $P(u)|_Q-vZ$ is nef. 
\\Suppose $v\ge 3-2u$. We find the Zariski decomposition of $P(u)|_S-vZ$. 
$$P(u,v)=(-e_1 - e_2 + (4 - 2u - v)\ell_1 + (4 - 2u)\ell_2)+(3 - 2u - v)(F_{21}+F_{22})=$$
$$=(-4 + 2u + v)e_1 + (-4 + 2u + v)e_2 + (4 - 2u - v)\ell_1 + (10 - 6u - 2v)\ell_2,$$
$$N(u,v)=-(3 - 2u - v)(F_{21}+F_{22})=(3 - 2u - v)(2\ell_2-e_1-e_2).$$
 We have that:
\\\\
\begin{minipage}{0.35\textwidth}
\begin{itemize}
    \item $P(u,v)\cdot e_1=4 - 2u - v$,
    \item $P(u,v)\cdot e_2=4 - 2u - v$,
    \item $P(u,v)\cdot F_{11}=6 - 4u - v$,
\end{itemize}
\end{minipage}
\hfill
\begin{minipage}{0.8\textwidth}
\begin{itemize}
    \item $P(u,v)\cdot F_{12}=6 - 4u - v$,
    \item $P(u,v)\cdot F_{21}=0$,
    \item $P(u,v)\cdot F_{22}=0$.
\end{itemize}
\end{minipage}
\\\\
To check when the divisor is pseudo-effective we should choose the strongest inequality from the following system:
$$\begin{cases}
v\le 4-2u,\\
v\le  6-4u,
\end{cases}\Rightarrow v\le  6-4u.
$$
We get $v\le 6-4u$. Note that for $v=6-4u$ we have that $P_4(u,6-4u)^2=0$ so $P(u)|_S-vZ$ is pseudo-effective if and only if  $v\le 6-4u$.
The Corollary~\ref{corollary:Kento-formula-Fano-threefold-surface-curve} gives us:
\begin{multline*}
S\big(W_{\bullet,\bullet}^{Q};Z\big)=\frac{3}{28}\int_{0}^{\frac{3}{2}}\int_{0}^{\infty}\mathrm{vol}\Big(P(u)\big\vert_{Q}-v\ell_1\Big)dvdu=\frac{3}{28}\int_{0}^{1}\int_{0}^{2-u}\Big(10 - 4u - 4v\Big)dvdv+\\
+\frac{3}{28}\int_{0}^{1}\int_{2-u}^{3-u}2(3-u-v)^2dvdu+\frac{3}{28}\int_{1}^{\frac{3}{2}}\int_{0}^{3-2u}\Big(8u^2 + 4uv - 32u - 8v + 30\Big)dvdu+\\
+\frac{3}{28}\int_{1}^{\frac{3}{2}}\int_{3-2u}^{6-4u}2(4-2u-v)(6-4u-v)dvdu=\frac{109}{112}.
\end{multline*}
So we obtained a contradiction with Corollary~\ref{corollary:Kento-formula-Fano-threefold-surface-curve}.
\\2). Suppose $Z\sim \ell_2$.
\\Take any $v\in \DR_{\ge 0}$. Abusing the notations suppose $P(u,v)$ is a positive part of the Zariski decomposition of $(-K_X-uS)|_S-vZ$,  $N(u,v)$ is a negative part of the Zariski decomposition of $(-K_X-uS)|_S-vZ$. 
\\If $u\in[0,1]$ then
$$P(u)\vert_{Q}-vZ\sim (-K_X-uQ)|_Q-vZ=-e_1 - e_2 + (3 - u)\ell_1 + (2-v)\ell_2.$$
Now  we find $v$ such that the divisor $P(u)\vert_{Q}-vZ$ is nef. We have that:
\\\\
\begin{minipage}{0.35\textwidth}
\begin{itemize}
    \item $P(u,v)\cdot e_1=1$,
    \item $P(u,v)\cdot e_2=1$,
    \item $P(u,v)\cdot F_{11}=1 - v$,
\end{itemize}
\end{minipage}
\hfill
\begin{minipage}{0.8\textwidth}
\begin{itemize}
    \item $P(u,v)\cdot F_{12}=1 - v$,
    \item $P(u,v)\cdot F_{21}=2 - u$,
    \item $P(u,v)\cdot F_{22}=2 - u$.
\end{itemize}
\end{minipage}
\\\\
Thus for $v\le 1$ we have that the divisor $P(u)|_Q-vZ$ is nef. 
\\Suppose $v\ge 1$. We find the Zariski decomposition of $P(u)|_Q-vZ$. 
$$P(u,v)=(-e_1 - e_2 + (3 - u)\ell_1 + (2-v)\ell_2)+(1-v)(F_{11}+F_{12})=$$
$$=(v - 2)e_1 + (v - 2)e_2 + (5 - u - 2v)\ell_1 + (2 - v)\ell_2,$$
$$N(u,v)=-(v-1)(F_{11}+F_{12})=-(v-1)(2\ell_1-e_1-e_2).$$
We have that:
\\\\
\begin{minipage}{0.35\textwidth}
\begin{itemize}
    \item $P(u,v)\cdot e_1=2 - v$,
    \item $P(u,v)\cdot e_2=2 - v$,
    \item $P(u,v)\cdot F_{11}=0$,
\end{itemize}
\end{minipage}
\hfill
\begin{minipage}{0.8\textwidth}
\begin{itemize}
    \item $P(u,v)\cdot F_{12}=0$,
    \item $P(u,v)\cdot F_{21}=-v + 3 - u$,
    \item $P(u,v)\cdot F_{22}=-v + 3 - u$.
\end{itemize}
\end{minipage}
\\\\
To check when the divisor is pseudo-effective we should choose the strongest inequality from the following system:
$$\begin{cases}
v\le 2,\\
v\le  3-u,
\end{cases}\Rightarrow v\le  2.
$$
We get $v\le 2$. Note that for $v=2$ we have that $P_2(u,2)^2=0$ so $P(u)|_S-vZ$ is pseudo-effective if and only if $v\le 2$.
\\For $u\in[1,3/2]$  and $v\in\mathbb{R}_{\geqslant 0}$ we have that:
$$P(u)\vert_{Q}-vZ\sim (-K_X-uQ)|_Q-vZ=-e_1 - e_2 + (4 - 2u)\ell_1 + (4 - 2u - v)\ell_2.$$
Now  we find $v$ such that the divisor $P(u)\vert_{Q}-vZ$ is nef. We have that:
\\\\
\begin{minipage}{0.35\textwidth}
\begin{itemize}
    \item $P(u,v)\cdot e_1=1$,
    \item $P(u,v)\cdot e_2=1$,
    \item $P(u,v)\cdot F_{11}=3 - 2u - v$,
\end{itemize}
\end{minipage}
\hfill
\begin{minipage}{0.8\textwidth}
\begin{itemize}
    \item $P(u,v)\cdot F_{12}=3 - 2u - v$,
    \item $P(u,v)\cdot F_{21}=3 - 2u$,
    \item $P(u,v)\cdot F_{22}=3 - 2u$.
\end{itemize}
\end{minipage}
\\\\
Thus for $v\le 3-2u$ we have that the divisor $P(u)|_Q-vZ$ is nef. 
\\Suppose $v\ge 3-2u$. We find the Zariski decomposition of $P(u)|_S-vZ$. 
$$P(u,v)=(-e_1 - e_2 + (4 - 2u)\ell_1 + (4 - 2u - v)\ell_2)+(3 - 2u - v)(F_{11}+F_{12})=$$
$$=(-4 + 2u + v)e_1 + (-4 + 2u + v)e_2 + (10 - 6u - 2v)\ell_1 + (4 - 2u - v)l_2,$$
$$N(u,v)=-(3 - 2u - v)(F_{11}+F_{12})=(3 - 2u - v)(2\ell_1-e_1-e_2).$$
 We have that:
 \\\\
\begin{minipage}{0.35\textwidth}
\begin{itemize}
    \item $P(u,v)\cdot e_1=4 - 2u - v$,
    \item $P(u,v)\cdot e_2=4 - 2u - v$,
    \item $P(u,v)\cdot F_{11}=0$,
\end{itemize}
\end{minipage}
\hfill
\begin{minipage}{0.8\textwidth}
\begin{itemize}
    \item $P(u,v)\cdot F_{12}=0$,
    \item $P(u,v)\cdot F_{21}=6 - 4u - v$,
    \item $P(u,v)\cdot F_{22}=6 - 4u - v$.
\end{itemize}
\end{minipage}
\\\\
We get $v\le 6-4u$. Note that for $v=6-4u$ we have that $P_4(u,6-4u)^2=0$ so $P(u)|_S-vZ$ is pseudo-effective if and only if $v\le 6-4u$.
The Corollary~\ref{corollary:Kento-formula-Fano-threefold-surface-curve} gives us:
\begin{multline*}
S\big(W_{\bullet,\bullet}^{Q};Z\big)=\frac{3}{28}\int_{0}^{\frac{3}{2}}\int_{0}^{\infty}\mathrm{vol}\Big(P(u)\big\vert_{Q}-v\ell_2\Big)dvdu=\frac{3}{28}\int_{0}^{1}\int_{0}^{1}\Big(2uv - 4u - 6v + 10\Big)dvdv+\\
+\frac{3}{28}\int_{0}^{1}\int_{1}^{2}2(v - 2)(v - 3 + u)dvdu+\frac{3}{28}\int_{1}^{\frac{3}{2}}\int_{0}^{3-2u}\Big(8u^2 + 4uv - 32u - 8v + 30\Big)dvdu+\\
+\frac{3}{28}\int_{1}^{\frac{3}{2}}\int_{3-2u}^{6-4u}2(4-2u-v)(6-4u-v)dvdu=\frac{89}{112}<1.
\end{multline*}
So we obtained a contradiction with Corollary~\ref{corollary:Kento-formula-Fano-threefold-surface-curve}.
\\3). Suppose $Z\sim \ell_1+\ell_2-e_1-e_2$.
\\Take any $v\in \DR_{\ge 0}$. Abusing the notations suppose $P(u,v)$ is a positive part of the Zariski decomposition of $(-K_X-uS)|_S-vZ$,  $N(u,v)$ is a negative part of the Zariski decomposition of $(-K_X-uS)|_S-vZ$. 
\\If $u\in[0,1]$ then:
$$P(u)\vert_{Q}-vZ\sim (-K_X-uQ)|_Q-vZ=(-1 + v)e_1 + (-1 + v)e_2 + (3 - u - v)\ell_1 + (2 - v)\ell_2.$$
Now  we find $v$ such that the divisor $P(u)\vert_{Q}-vZ$ is nef. We have that:
\\\\
\begin{minipage}{0.35\textwidth}
\begin{itemize}
    \item $P(u,v)\cdot e_1=-v + 1$,
    \item $P(u,v)\cdot e_2=-v + 1$,
    \item $P(u,v)\cdot F_{11}=1$,
\end{itemize}
\end{minipage}
\hfill
\begin{minipage}{0.8\textwidth}
\begin{itemize}
    \item $P(u,v)\cdot F_{12}=1$,
    \item $P(u,v)\cdot F_{21}=2 - u$,
    \item $P(u,v)\cdot F_{22}=2 - u$.
\end{itemize}
\end{minipage}
\\\\
Thus for $v\le 1$ we have that the divisor $P(u)|_Q-vZ$ is nef.
\\Suppose $v\ge 1$. We find the Zariski decomposition of $P(u)|_Q-vZ$. 
$$P(u,v)=((-1 + v)e_1 + (-1 + v)e_2 + (3 - u - v)\ell_1 + (2 - v)\ell_2)+(1-v)(e_1+e_2)=$$$$=(3 - u - v)\ell_1 + (2 - v)\ell_2,$$
$$N(u,v)=-(v-1)(e_1+e_2).$$
We have that:
\\\\
\begin{minipage}{0.35\textwidth}
\begin{itemize}
    \item $P(u,v)\cdot e_1=0$,
    \item $P(u,v)\cdot e_2=0$,
    \item $P(u,v)\cdot F_{11}=2 - v$,
\end{itemize}
\end{minipage}
\hfill
\begin{minipage}{0.8\textwidth}
\begin{itemize}
    \item $P(u,v)\cdot F_{12}=2 - v$,
    \item $P(u,v)\cdot F_{21}=-v + 3 - u$,
    \item $P(u,v)\cdot F_{22}=-v + 3 - u$.
\end{itemize}
\end{minipage}
\\\\
To check when the divisor is pseudo-effective we should choose the strongest inequality from the following system:
$$\begin{cases}
v\le 2,\\
v\le  3-u,
\end{cases}\Rightarrow v\le  2.
$$
We get $v\le 2$. Note that for $v=2$ we have that $P_2(u,2)^2=0$ so $P(u)|_S-vZ$ is pseudo-effective until $v$ satisfies $v\le 2$.
\\For $u\in[1,3/2]$  and $v\in\mathbb{R}_{\geqslant 0}$ we have that:
$$P(u)\vert_{Q}-vZ\sim (-K_X-uQ)|_Q-vZ=(v-1 )e_1 + ( v-1)e_2 + (4 - 2u - v)\ell_1 + (4 - 2u - v)\ell_2.$$
Now  we find $v$ such that the divisor $P(u)\vert_{Q}-vZ$ is nef. We have that:
\\\\
\begin{minipage}{0.35\textwidth}
\begin{itemize}
    \item $P(u,v)\cdot e_1=-v + 1$,
    \item $P(u,v)\cdot e_2=-v + 1$,
    \item $P(u,v)\cdot F_{11}=3 - 2u$,
\end{itemize}
\end{minipage}
\hfill
\begin{minipage}{0.8\textwidth}
\begin{itemize}
    \item $P(u,v)\cdot F_{12}=3 - 2u$,
    \item $P(u,v)\cdot F_{21}=3 - 2u$,
    \item $P(u,v)\cdot F_{22}=3 - 2u$.
\end{itemize}
\end{minipage}
\\\\
Thus, for $v\le 1$ we have that the divisor $P(u)|_Q-vZ$ is nef. 
\\Suppose $v\ge 3-2u$. We find the Zariski decomposition of $P(u)|_S-vZ$. 
$$P(u,v)=((v-1 )e_1 + ( v-1)e_2 + (4 - 2u - v)\ell_1 + (4 - 2u - v)\ell_2)+(1- v)(e_1+e_2)=$$$$=(4 - 2u - v)\ell_1 + (4 - 2u - v)\ell_2,$$
$$N(u,v)=-(3 - 2u - v)(e_1+e_2).$$
 We have that:
 \\\\
\begin{minipage}{0.35\textwidth}
\begin{itemize}
    \item $P(u,v)\cdot e_1=0$,
    \item $P(u,v)\cdot e_2=0$,
    \item $P(u,v)\cdot F_{11}=4 - 2u - v$,
\end{itemize}
\end{minipage}
\hfill
\begin{minipage}{0.8\textwidth}
\begin{itemize}
    \item $P(u,v)\cdot F_{12}=4 - 2u - v$,
    \item $P(u,v)\cdot F_{21}=4 - 2u - v$,
    \item $P(u,v)\cdot F_{22}=4 - 2u - v$.
\end{itemize}
\end{minipage}
\\\\
We get $v\le 4-2u$. Note that for $v=4-2u$ we have that $P(u,4-2u)^2=0$ so $P(u)|_S-vZ$ is pseudo-effective until $v$ satisfies $v\le 4-2u$.
The Corollary~\ref{corollary:Kento-formula-Fano-threefold-surface-curve} gives us:
\begin{multline*}
S\big(W_{\bullet,\bullet}^{Q};Z\big)=\frac{3}{28}\int_{0}^{\frac{3}{2}}\int_{0}^{\infty}\mathrm{vol}\Big(P(u)\big\vert_{Q}-v\ell_2\Big)dvdu=\frac{3}{28}\int_{0}^{1}\int_{0}^{1}\Big(2uv - 4u - 6v + 10\Big)dvdv+\\
+\frac{3}{28}\int_{0}^{1}\int_{1}^{2}2(v - 2)(-3 + u + v)dvdu+\frac{3}{28}\int_{1}^{\frac{3}{2}}\int_{0}^{1}\Big(2(2u - 3)(2u + 2v - 5)\Big)dvdu+\\
+\frac{3}{28}\int_{1}^{\frac{3}{2}}\int_{1}^{4-2u}2(-4 + 2u + v)^2dvdu=\frac{47}{56}<1.
\end{multline*}
So we obtained a contradiction with Corollary~\ref{corollary:Kento-formula-Fano-threefold-surface-curve}.
\\\\ Thus, for any $G$-invariant curve $Z\subset Q$ such that $Z\ne E_C\vert_{Q}$ we have $S\big(W_{\bullet,\bullet}^{Q};Z\big)<1$ which is impossible by Corollary~\ref{corollary:Kento-formula-Fano-threefold-surface-curve}.
\\\\$\boxed{2}$ Suppose $Z=E_C\vert_{Q}$, then
\begin{multline*}
S\big(W_{\bullet,\bullet}^{Q};Z\big)=\frac{3}{28}\int_{0}^{\frac{3}{2}}\Big(P(u)\cdot P(u)\cdot Q\Big)\mathrm{ord}_Z\Big(N(u)\big\vert_{Q}\Big)du+\frac{3}{28}\int_{0}^{\frac{3}{2}}\int_{0}^{\infty}\mathrm{vol}\Big(P(u)\big\vert_{Q}-vZ\Big)dvdu=\\
=\frac{3}{28}\int_{1}^{\frac{3}{2}}(u-1)\big((4-2u)\pi^*(H)-E_L\big)^2\cdot\big(2\pi^*(H)-E_C\big)du+\frac{3}{28}\int_{0}^{\frac{3}{2}}\int_{0}^{\infty}\mathrm{vol}\Big(P(u)\big\vert_{Q}-vZ\Big)dvdu=\\
=\frac{3}{28}\int_{1}^{\frac{3}{2}}(u-1)\big(2(4-2u)^2-2\big)du+\frac{3}{28}\int_{0}^{\frac{3}{2}}\int_{0}^{\infty}\mathrm{vol}\Big(P(u)\big\vert_{Q}-vZ\Big)dvdu=\\
=\frac{5}{224}+\frac{3}{28}\int_{0}^{\frac{3}{2}}\int_{0}^{\infty}\mathrm{vol}\Big(P(u)\big\vert_{Q}-v(\ell_1+2\ell_2)\Big)dvdu\leqslant\\
\leqslant\frac{5}{224}+\frac{3}{28}\int_{0}^{\frac{3}{2}}\int_{0}^{\infty}\mathrm{vol}\Big(P(u)\big\vert_{Q}-v\ell_1\Big)dvdu=\frac{5}{224}+\frac{109}{112}=\frac{223}{224}<1.
\end{multline*}
So we obtained a contradiction with Corollary~\ref{corollary:Kento-formula-Fano-threefold-surface-curve}. This completes the~proof of the~lemma.
\end{proof}
\begin{corollary}\label{5}
The curve $\pi(Z)$ is not a line that intersect $C$.
\end{corollary}
\begin{proof}
Recall from Section \ref{3} that there are exactly 3 $G$-invariant lines in $\DP^3$ that intersect the curve $C$.
These are the lines $L_{12}$, $L_{34}$, $L_{56}$. We have that $L_{12}\subset Q_2\cap Q_3$, $L_{34}\subset Q_1\cap Q_2$, $L_{34}\subset Q_1 \cap Q_3$. Thus, the lemma above gives us the result.
\end{proof}
\noindent By Lemma \ref{lemma144}, one has $\alpha_{G,Z}(X)<\frac{3}{4}$.
Thus, by lemma \cite[Lemma 1.4.1]{Fano21} and it's proof, there is a~$G$-invariant effective $\mathbb{Q}$-divisor $D$ on the~threefold $X$
such that $D\sim_{\mathbb{Q}}-K_X$ and $Z\subseteq\mathrm{Nklt}(X,\lambda D)$ for some~positive rational number $\lambda<\frac{3}{4}$.
\begin{lemma}
\label{lemma:3-12-Nklt}
Let $S$ be an~irreducible surface in $X$. Suppose that $S\subset\mathrm{Nklt}(X,\lambda D)$.
Then either $S\in|\pi^*(2H)-E_C|$ and $S$ is $G$-invariant or $S=E_L$. 
\end{lemma}

\begin{proof}
We have $D\sim_{\mathbb{Q}} 4\pi^*(H)-E_C-E_L$ and $\lambda<\frac{3}{4}$. By assumption we have $D=aS+\Delta$ where $a\in\DQ$ such that $a\ge \frac{1}{\lambda}>\frac{4}{3}$ and $\Delta$ is an effective $\DQ$ divisor on $X$ whose support does not contain $S$.
\\Assume $S=E_C$. Then we get 
$$\pi^*(4H)-E_C-E_L\sim_{\DQ} aE_C+\Delta\Rightarrow\Delta \sim_{\DQ} \pi^*(4H)-(1+a)E_C-E_L=R-(a-1)E_C$$
-contradiction.
\\Assume $S\ne E_L$, $S\ne E_C$ then $\pi(S)\subset \DP^3$ is the surface of some degree $d$. We have that:
$$4H\sim_{\DQ}a\pi(S)+\pi(\Delta)\Rightarrow 4\ge ad\Rightarrow d=1\text{ or }d=2.$$
The latter holds since $a>\frac{4}{3}$. Then $S$ is given by $$S\sim_{\DQ}d\pi^*(H)-m_LE_L-m_CE_C.$$
 By Corollary \ref{4} we know that the cone of effective divisors is generated by $E_L$, $E_C$, $R$, $H-E_L$ so we have that:
$$\Delta \sim_{\DQ}\pi^*(4H)-E_C-E_L-a(\pi^*(4H)-(1+a)E_C-E_L=R-(a-1)E_C)\sim $$
$$\sim a_1E_L+a_2E_C+a_3(\pi^*(4H)-2E_C-E_L)+a_4(\pi^*(H)-E_L)$$
for $a_1\ge0,\text{ }a_2\ge0,\text{ }a_3\ge0,\text{ }a_4\ge0$, which gives us a system of equations:
\begin{equation}
\begin{cases}
-4a_3 - a_4 - d + 4=0,\\
-1 + m_L - a_1 + a_3 + a_4=0,\\ -1 + m_C - a_2 + 2a_3=0,\\
a_1\ge0,\text{ }a_2\ge0,\text{ }a_3\ge0,\text{ }a_4\ge0, a>4/3
\end{cases}
\label{eq:2}
\end{equation}
Thus, if $d=2$ then $(m_L,m_C)=(0,1)$ or $(m_L,m_C)=(1,1)$ so we have the following options:
\begin{itemize}
     \item $m_C= 1$ and $m_L= 1$. This gives us the linear system $|S|=|2\pi^*(H)-m_LE_L-m_CE_C|$ which on $\DP^3$ corresponds to the linear system of quadrics which contain a line $L$  and a cubic $C$. But this is impossible since by assumption $L$ and $C$ so not intersect. Thus, this linear system does not contain effective divisors.
    \item $m_C=1$, $m_L=0$. Suppose $S\in|2\pi^*(H)-E_C|$ and $S$ is not $G$-invariant. We have that $S\subset\mathrm{Nklt}(X,\lambda D)$ where $D$ is a $G$-invariant $\DQ$-divisor. We can write $D$ as $D=\sum_i a_iD_i$ where $D_i$-s are irreducible components of $D$, $a_i\in \DQ_{>0}$ and we have $S=D_i$ for some $i$. We assumed that $S$ is not $G$-invariant thus if we take a non-trivial element $g\in G$ we will have that $S'=g(S)$ is $D_i$ which is one of the components of $D$ for $i\ne j$. Moreover $a_j=a_j=a$ since $D$ is $G$-invariant so we can write:
     $$\pi^{*}(4H)-E_C-E_L\sim_{\DQ}2a(2\pi^*(H)-E_C)+\Delta.$$
     Where $\Delta$ is an effective $\DQ$-divisor. Thus:
     $$4(1-a)\pi^*(H)+(-1 + 2a)E_C - E_L\sim_{\DQ}\Delta.$$
     By Corollary \ref{4} we know that the cone of effective divisors is generated by $E_L$, $E_C$, $R$, $H-E_L$ so we can write $\Delta$ as:
     $$\Delta=a_1E_L+a_2E_C+a_3(4\pi^*(H)-2E_C-E_L)+a_4(H-E_L).$$
     For $a_1\ge0,\text{ }a_2\ge0,\text{ }a_3\ge0,\text{ }a_4\ge0$. Solving the system of equations on coefficients we get that it has no solutions.
    \\Suppose $S\in|2\pi^*(H)-E_C|$ and $S$ is  $G$-invariant. This case is possible.
    \end{itemize}
\noindent Similarly, if $d=1$ then $(m_L,m_C)=(0,0)$ or $(m_L,m_C)=(1,0)$ so we have the following options:
 \begin{itemize}
     \item $m_C=0$, $m_L=1$. Suppose $S\in|\pi^*(H)-E_L|$. We have that $S\subset\mathrm{Nklt}(X,\lambda D)$ where $D$ is a $G$-invariant $\DQ$-divisor. We can write $D$ as $D=\sum_i a_iD_i$ where $D_i$-s are irreducible components of $D$, $a_i\in \DQ_{>0}$ and we have $S=D_i$ for some $i$. Note that $S\in |\pi^*(H)-E_L|$ where $|\pi^*(H)-E_L|$ does not have $G$-invariant elements. Take a non-trivial element $g\in G$. We have that $S'=g(S)$ is $D_i$ which is one of the components of $D$ for $i\ne j$. Moreover $a_j=a_j=a$ since $D$ is $G$-invariant so we can write:
     $$\pi^{*}(4H)-E_C-E_L\sim_{\DQ}2a(\pi^*(H)-E_L)+\Delta.$$
     Where $\Delta$ is an effective $\DQ$-divisor. Thus:
     $$(4-2a)\pi^{*}(H)-E_C-(1-2a)E_L\sim_{\DQ}\Delta.$$
     By Corollary \ref{4} we know that the cone of effective divisors is generated by $E_L$, $E_C$, $R$, $H-E_L$ so we can write $\Delta$ as:
     $$\Delta=a_1E_L+a_2E_C+a_3(4\pi^*(H)-2E_C-E_L)+a_4(\pi^*(H)-E_L).$$
     For $a_1\ge0,\text{ }a_2\ge0,\text{ }a_3\ge0,\text{ }a_4\ge0$.  Solving the system of equations on coefficients we get that it has no solutions.
     \\
     \item $m_C=0$, $m_L=0$. Suppose $S\in|\pi^*(H)|$. We have that $S\subset\mathrm{Nklt}(X,\lambda D)$ where $D$ is a $G$-invariant $\DQ$-divisor. We can write $D$ as $D=\sum_i a_iD_i$ where $D_i$-s are irreducible components of $D$, $a_i\in \DQ_{>0}$ and we have $S=D_i$ for some $i$. Note that $S\in |\pi^*(H)|$ where $|\pi^*(H)|$ does not have $G$-invariant elements. Take a non-trivial element $g\in G$. We have that $S'=g(S)$ is $D_i$ which is one of the components of $D$ for $i\ne j$. Moreover $a_j=a_j=a$ since $D$ is $G$-invariant so we can write:
     $$\pi^{*}(4H)-E_C-E_L\sim_{\DQ}2a\pi^*(H)+\Delta.$$
     Where $\Delta$ is an effective $\DQ$-divisor. Thus:
     $$(4-2a)\pi^{*}(H)-E_C-E_L\sim_{\DQ}\Delta.$$
     By Corollary \ref{4} we know that the cone of effective divisors is generated by $E_L$, $E_C$, $R$, $H-E_L$ so we can write $\Delta$ as:
     $$\Delta=a_1E_L+a_2E_C+a_3(4\pi^*(H)-2E_C-E_L)+a_4(H-E_L).$$
    For $a_1\ge0,\text{ }a_2\ge0,\text{ }a_3\ge0,\text{ }a_4\ge0$.  Solving the system of equations on coefficients we get that it has no solutions. 
 \end{itemize}
 We see that we excluded all options except  $S\in|\pi^*(2H)-E_C|$ and $S$ is $G$-invariant or $S=E_L$.
\end{proof}
\begin{corollary}\label{pi(Z)isnotsurf}
$\pi(Z)$ is not a surface in $\mathrm{Nklt}(X,\lambda D)$.
\end{corollary}
\begin{corollary}
\label{corollary:3-12-EC}
One has $Z\not\subset E_C$.
\end{corollary}

\begin{proof}
Suppose that $Z\subset E_C$. Observe that $\pi(Z)$ is not a~point, since $\mathbb{P}^3$ does not have $G$-fixed points by Lemma~\ref{fixedpoints}.
Hence, we see that $\pi(Z)$ is the twisted cubic $C$.
\\
Let $S$ be a~general fiber of $\eta$. Then $S\cdot Z\ge 3$, which contradicts Lemma \ref{CorollaryA15}.
\end{proof}

\begin{lemma}
\label{lemma:3-12-curves-line}
The curve $\pi(Z)$ is the~line.
\end{lemma}
\begin{proof}
Let $\overline{D}=\pi(D)$, $\overline{Z}=\pi(Z)$. We see that $\overline{Z}$ is a~$G$-invariant curve in $\mathbb{P}^3$ such that such that Z is not contained in a $G$-invariant surface in $|2\pi^*(H)-E_C|$ (by Lemma \ref{lemma:3-12-Z-in-Q}), $Z\not\subset E_L$ (by Lemma \ref{lemma:3-12-Z-in-E}) and $Z\not\subset E_C$ (by Lemma \ref{corollary:3-12-EC}). Then $\overline{Z}\subset\mathrm(\DP^3,\lambda \overline{D})$ and $\overline{Z}$ is not contained in any surface contained in $\mathrm{Nklt}(\DP^3,\lambda \overline{D}) $ by Lemma \ref{lemma:3-12-Nklt}. Now we apply  Lemma \ref{CorollaryA13} and get that $\MO_{\DP^3}(1)\cdot \overline{Z}\le 1$. Thus $\MO_{\DP^3}(1)\cdot \overline{Z}= 1$ so $\pi(Z)$ is a line.
\end{proof}

\begin{corollary}
Such irreducible curve $Z$ does not exist.
\end{corollary}
\begin{proof}
By Lemma \ref{lemma:3-12-curves-line} we know that $\pi(Z)$ is a line. We have that $\pi(Z)\ne L$ (by Lemma \ref{lemma:3-12-Z-in-E}),
$\pi(Z)$ is not one of the $G$-invariant lines which does not intersect $C$ and $\pi(Z)\ne L$ (by Lemma \ref{lemma:3-12-Z-lines}) and  $\pi(Z)$ is not one of the $G$-invariant lines which intersect $C$ (by Corollary \ref{5}). So such irreducible curve $Z$ does not exist.
\end{proof}
\noindent This completes the proof of Main Theorem.

$ $ \\\newline\newline
\it{Elena Denisova}
\newline
\textnormal{The University of Edinburgh, Edinburgh, Scotland}
\newline
\textnormal{\texttt{e.denisova@sms.ed.ac.uk}}
\end{document}